\documentclass{amsart}
\usepackage{amssymb}
\theoremstyle{plain}
\newtheorem{lem}{Lemma}[section]
\newtheorem{thm}{Theorem}[section]
\newtheorem{cor}{Corollary}
\newtheorem{prop}{Proposition}[section]

\theoremstyle{definition}
\newtheorem{exa}{Example}

\newtheorem{dfn}{Definition}[section]
\numberwithin{equation}{section}

\newcommand{\bpf}{\begin{proof}}
\newcommand{\epf}{\end{proof}}
\input amssym.def
\input amssym
\begin{document}
%Dear professor R. Shankar, Miodrag Mateljevi\' c
\author{Miodrag Mateljev\textrm{i\'{c}}}
\address{Faculty of mathematics, Univ. of Belgrade, Studentski Trg 16, Belgrade, Serbia}
\email{\rm miodrag@matf.bg.ac.rs}
\title[Bounds for the modulus of the derivatives of  harmonic mappings]{The
lower bound for the modulus of the derivatives and Jacobian of
harmonic injective  mappings}
%Invertible harmonic mappings  and  of harmonic mappings
\maketitle

\thanks{Research partially supported by MNTRS, Serbia,  Grant No.
174 032}

This is a very rough version.

\bigskip
\bigskip

\begin{abstract}
We give  the lower bound for the modulus of the radial derivatives
and Jacobian of harmonic injective mappings from the unit ball
onto convex domain  in plane and space.  As an application we show
%that  if, in addition,  $f$ is univalent harmonic K-qc gradient
%mapping, then $f$ is
co-Lipschitz property of some classes of qch mappings. We also
review related results in planar case using some novelty.
\end {abstract}

Throughout the paper we denote by $\Omega $, $G$ and $D$ open subset of $%
\mathbb{R}^{n}$, $n\geq 1$.

Let $B^{n}(x,\ r)=\{z\in \mathbb{R}^{n}:|z-x|<r\},\ S^{n-1}(x,\
r)=\partial B^{n}(x,\ r)$ (abbreviated $S(x,\ r)$) and let
$\mathbb{B}^{n},\ \mathbb{S}=\mathbb{S}^{n-1}$ stand for the unit
ball and the unit sphere in $\mathbb{R}^{n}$, respectively. In
particular,  by $\,\mathbb{D}\,$ we denote the unit disc
$\mathbb{B}^2$  and $\,\mathbb{T}=\partial \mathbb{D}\,$ we denote
the unit circle $\mathbb{S}^1$  in the complex plane.

For a
domain $D$ in $\mathbb{R}^{n}$ with non-empty boundary, we define
the
distance function $d=d_D=\mathrm{dist}(D)$ by $d(x)=d(x;\partial D)=\mathrm{%
dist}(D)(x)=\inf \{|x-y|:y\in \partial D\}$; and if $f$ maps $D$ onto $%
D^{\prime }\subset \mathbb{R}^{n}$, in some settings it is
convenient to use short notation $d^{\ast}=d^{\ast}(x)=d_{f}(x)$
for $d(f(x);\partial D^{\prime })$. It is clear that
$d(x)=\mathrm{dist}(x,\ D^{c})$, where $D^{c}$ is the complement
of $D$ in $\mathbb{R}^{n}$.
Let $G$ be an open set in ${\mathbb{R}}^{n}$. A mapping $f:G\rightarrow {%
\mathbb{R}}^{m}$ is differentiable at $x\in G$ if there is a linear mapping $%
f^{\prime }(x):{\mathbb{R}}^{n}\rightarrow {\mathbb{R}}^{n}$,
called the derivative of $f$ at $x$, such that
\begin{equation*}
f(x+h)-f(x)=f^{\prime }(x)h+|h|\varepsilon (x,h)
\end{equation*}%
where $\varepsilon (x,\ h)\rightarrow 0$ as $h\rightarrow 0$.
%  say that it is differentiable at $x \in G$  if there is a linear mapping denote the derivative of $f$ at $x$  by $f'(x)$ and
For a vector-valued function $f:G\rightarrow {\mathbb{R}}^{n}$, where $%
G\subset {\mathbb{R}}^{n}$, is a domain, we define
\begin{equation*}
|f^{\prime }(x)|=\max\limits_{|h|=1}|f^{\prime }(x)h|\,\quad
\text{and}\quad l(f^{\prime }(x))=\min\limits_{|h|=1}|f^{\prime
}(x)h|\,,
\end{equation*}%
when $f$ is differentiable at $x\in G\,.\ $   Occasionally  we use
the notation  $\Lambda_f(x)$  and  $\lambda_f(x)$  instead of
$|f^{\prime }(x)|$  and   $\ell(f^{\prime }(x)$ (in particular in
planar case) respectively.

For $x \in R^n$, we   use notation $r=|x|$.  We say  that Jacobian
$J$ of mapping on a domain $\Omega$   satisfies minimum principle
if for every compact $F \subset \Omega$
 we have  $\inf_ F J\geq inf _{\partial F} J$.
 A   $C^1$ (in particular diffemorphisam)  mapping $f: \Omega\rightarrow \Omega^*$ is  $K$-qc iff

\begin{equation}
|f'(x)|^n/K \leq |J(x,f)|\leq K \, \ell((f'(x))^n
\end{equation}
holds for every $x\in \Omega$.   For   $\xi \in \mathbb{S}$,
define
$$ h_b(\xi) = h^*(\xi)=\lim_{r\rightarrow 1}
h(r \xi)$$ when this limit exists.

The directional derivative of a scalar function  $f(\bold{x}) =
f(x_1, x_2, \ldots, x_n)$   along a vector $\bold{v} = (v_1,
\ldots, v_n)$ is the function defined by the limit

   $$ \nabla_{\bold{v}}{f}(\bold{x}) = \lim_{h \rightarrow 0}{\frac{f(\bold{x} + h\bold{v}) - f(\bold{x})}{h}}.$$

If the function $f$ is differentiable at $x$, then the directional
derivative exists along any vector $v$, and one has

    $$\nabla_{\bold{v}}{f}(\bold{x}) = \nabla f(\bold{x}) \cdot \bold{v}\,,$$
where the $\nabla$ on the right denotes the gradient and $\cdot$
is the dot product. Intuitively, the directional derivative of f
at a point $x$ represents the rate of change of $f$ with respect
to time when it is moving at a speed and direction given by $v$.
Instead of  $\nabla_{\bold{v}}{f}$ we also write
$D_{\bold{v}}{f}$.   If   $v= \frac{x}{|x|}$, $x\neq 0$, $g(t)=
f(x +t v)$  and $D_{\bold{v}}{f}(x)$ exists, we define $\partial_r
f(x)= D_{\bold{v}}{f}(x)=g'(0)$.

Let $\Omega \in \mathbb{R}^{n}$ and $\mathbb{R}^{+}=[0,\ \infty )$ and $f,\
g:\Omega \rightarrow \mathbb{R}^{+}$. If there is a positive constant $c$
such that $f(x)\leq c\,g(x)\,,\ x\in \Omega \,$, we write $f\preceq g$ on $%
\Omega $. If there is a positive constant $c$ such that

\begin{equation*}
\frac{1}{c}\,g(x)\leq f(x)\leq c\,g(x)\,,\quad x\in \Omega \,,
\end{equation*}%
we write $f\approx g$ (or $f\asymp g$ ) on $\Omega $.

O.Martio \cite{OM1} observed that, every quasiconformal harmonic
mapping of the unit disk onto itself is co-Lipschitz. Then the
subject was intensively studied by the participants of Belgrade
Analysis Seminar, see for example  \cite{mm.fil12a,k.jdam,aaa,MP}
and the literature cited there. In particular  Kalaj and
Mateljevi\'c, shortly KM-approach, study lower bound of  Jacobian.
The corresponding results for harmonic maps between surfaces were
obtaind previously by  Jost and Jost-Karcher  \cite{jost2,jost}.
We refer to this results shortly as JK- result (approach).
%%%%%
%{mm.unpub2} {aaa} harmonic mappings to
%%%%%
Recently  Iwaniec has communicated the proof of
Rado-Kneser-Choquet  theorem (shortly Theorem RKC), cf.
\cite{Iw14}, Iwaniec- Onninen  cf. \cite{IwOn}. We refer to this
communication shortly as IwOn-approach.
 It seems that there is some overlap between  KM- results  with  \cite{Iw14,IwOn} and
\cite{jost2,jost} (we will shortly describe it in Section
\ref{sec.pl}). Note only here that in planar case  JK- result  is
reduced to Theorem RKC.

%%%% {KaMpacific,aams,aaa}
The author has begun to consider  harmonic functions  in the space
roughly since 2006 trying to  generalize  theory in the plane, cf
\cite{mm.unpub2,akm,mavu,KaMpacific,aams,aaa}.\\ He realized some
differences  between theory in the  plane and  space  and  some
difficulties to develop  the space theory.    It was observed that
gradient mappings of harmonic functions are good candidate for
generalization of the planar theory  to space  outlining   some
ideas and asking  several open problems on the Belgrade Analysis
seminar. Having studied   Iwaniec's lecture \cite{Iw14} recently
%Iwaniec-Onninen  paper  \cite{IwOn}.  gave
the author has found  an additional motivation to investigate  in
this direction. For the present state  see also the recent  arXiv
papers of
%Iwaniec-Onninen  \cite{IwOn},
Astala-Manojlovi\' c \cite{ast.ma},  Bo\v zin-Mateljevi\'c
\cite{BoMa2} and Mateljevi\'c \cite{rckm0}.
%For example, in  \cite{akm} with coauthors  Lipschitz continuity of harmonic quasiregular maps on the unit
%ball in $\mathbb{R}^{n}$ is considered as a generalization of planar  theory.

Suppose that $\,F\,$ is
%a euclidean harmonic
mapping from  a domain $G\subset \mathbb{R}^n $ (in particular,
from the unit ball $\mathbb{B}\subset \mathbb{R}^n$) onto a
bounded convex  domain $D=F(G)$.
%$D=F(\mathbb{B})$.
To every  $a\in
\partial D$  we associate a nonnegative
%harmonic
function $u=\underline{u}_a=\underline{F}^a$.  Since  $D$ is
convex, for $a\in
\partial D$, there is a  supporting hyper-plane (a subspace of dimension $n - 1$) $\Lambda_a$ defined by $\Lambda_a=\{w \in \mathbb{R}^n: (w-a,n_a)=0 \}$, where
$n=n_a\in T_a\mathbb{R} ^n$  is  a  unit vector
%unimodular complex number
such that $ (w-a,n_a)\geq 0$ for every $w\in \overline{D}$. Define $u(z)=\underline{F}^a(z)=(F(z)-a,n_a)$, cf.
\cite{revroum01,napoc1,mm.fil12a,kalaj.thesis}. Our approuch here is also
based on  function $\underline{F}^a$.

%Then
% ?? The subject was also intensively studied by the participants of  Belgrade Analysis Seminar  (in particular by Kalaj and
%Mateljevi\'c, shortly KM-approach), see for example \cite{mm.fil12a} and literature  cited there.

We  provide explicit interior
lower bounds on the Jacobian in terms of the regularity of the
domains and the boundary map  in  Section \ref{sec.pl}. For the convenience of the  reader in  Section   \ref{app} we mainly collect  some results
which we used in Section \ref{sec.pl}  and give a few additional results.

In  Section   \ref{app2},  we outline short review of results from
\cite{jost2}  for  Harmonic Maps Between surfaces concerning lower
bounds on the Jacobian,  and  the existence of harmonic
diffeomorphisms which solve a Dirichlet problem.

In Section  \ref{s.space},  estimates for the modulus of the
derivatives of harmonic univalent mappings in space  are given.

In Section \ref{grad3}  we  generalize and develop the arguments
used in planar theory of harmonic mappings to gradient mappings of
harmonic functions in domains of $\mathbb{R}^3$. For example,  we
can consider the proof  of Theorem  \ref{3clip} (which is not
based on an approximation  argument) as a suitable generalization
of the proof of Theorem \ref{1clip}.

% It seems that there is some overlap between results presented in  Section \ref{grad3}   and,concerning  univalent harmonic K-qc gradient mapping,
\section{Estimates for the modulus of the derivatives of
harmonic univalent planar mappings from below}\label{sec.pl}
\subsection{}For univalent harmonic maps between surfaces, estimates of
Jacobian from below in terms of the geometric data involved  are
given in  Jost \cite{jost2} (see  Corollary 8.1, Theorem 8.1) and
for univalent euclidean harmonic maps  in
\cite{kalaj.thesis,kal.studia,mm.fil12a,aaa}.
%Mateljevi\'c
In this subsection, we consider convex  codomains  and give short
review of a few result from \cite{napoc1,revroum01}. It  seems
that Theorem \ref{1clip} is a new result. For a function $h$, we
use notation $\partial h=\frac{1}{2}(h_{x}^{\prime
}-ih_{y}^{\prime })$ and $\overline{\partial
}h=\frac{1}{2}(h_{x}^{\prime }+ih_{y}^{\prime })$; we also use
notations $Dh$ and $\overline{D}h$ instead of $\partial h$ and
$\overline{\partial }h$ respectively when it seems convenient.

Throughout this paper, if $h$  is a complex  harmonic function on
simple connected planar domain, we will write $h$ in the form
$h=f+\overline{g}$, where $f$ and $g$ are    holomorphic. Note
that   every complex valued harmonic function $h$ on simply
connected domain  $D$ is of this form.

Recall by $\,\mathbb{D}\,$ we denote the unit disc and
$\,\mathbb{T}=\partial \mathbb{D}\,$ we denote the unit circle, and we use
notation $z=re^{i\theta }$. For a function $h$ we denote by $h_{r}^{\prime
},\ h_{x}^{\prime }$ and $h_{y}^{\prime }$ (or sometimes by $\partial _{r}h$%
, $\partial _{x}h$ and $\partial _{x}h$) partial derivatives with respect to
$r$, $x$ and $y$ respectively. Let $h=f+\overline{g}$ be harmonic, where $f$
and $g$ are analytic. Then $\,\partial h=f^{\prime }$, $%
h_{r}^{\prime }=f_{r}^{\prime }+\overline{g_{r}^{\prime }}$, $f_{r}^{\prime
}=f^{\prime }(z)e^{i\theta }$ and $J_{h}=|f^{\prime }|^{2}-|g^{\prime }|^{2}$%
. If $h$ is univalent, then $|g^{\prime }|<|f^{\prime }|$ and therefore $%
|h_{r}^{\prime }|\leq |f_{r}^{\prime }|+|g_{r}^{\prime }|$ and $%
|h_{r}^{\prime }|<2|f^{\prime }|$.

\begin{thm}[\cite{napoc1}]\label{t1}
Suppose  that\\
{\rm (a)}    $\,h\,$ is    a euclidean harmonic mapping from
$\mathbb{D}$ onto  a  bounded convex  domain  $D=h(\mathbb{D})$,
which
contains the  disc $\,B(h(0);R_0)\,$.  Then \\
$(i.1)$  $d(h(z),\partial D) \geq (1-|z|)R_0/2 $, \, $z\in \mathbb{D}$.\\
$(i.2)$  Suppose  that $ \omega=h^*(e^{i\theta})$  and
$h^*_r=h'_r(e^{i\theta})$ exist at a point $ e^{i\theta}\in
\mathbb{T}$,  and  there exists the unit inner normal $n=n_{
\omega}$ at
$ \omega=h^*(e^{i\theta})$ with respect to $\partial D$.\\
%If  $a=h^*(e^{i\theta})$ and  $h^*_r=h'_r(e^{i\theta})$ exist,
Then  $(h^*_r, n)\geq c_0$,
% where   $n=n_a$  is unit inner normal at $a$ with respect to $\partial D$
% a unimodular complex number such that
% $(w-a,n_a)\geq 0$ for every $w\in \overline{D}$
where    $c_0=\frac{R_0}{2}$.\\
$(i.3)$ In addition to  the hypothesis  stated in the item $i.2)$,
suppose  that  $h_b'$   exists at the  point $e^{i\theta}$.   Then
$ |J_h| = |(h_r^*, N)|= \big|(h_r^*, n)\big||N|\geq c_0 |N|$,
where $N=i\,h_b'$   and  the Jacobian is computed at the  point
$e^{i\theta}$ with respect to the polar coordinates.

$(i.4)$    If in addition to the hypothesis  {\rm (a)}   suppose that  $\,h\,$ is an euclidean univalent harmonic
mapping from an open set $G$ which contains $\,\overline {\mathbb{D}}\,$. Then    $|f^{\prime }|\geq \frac{R}{4}$ on $\,\mathbb{D}$.
%$
\end{thm}
\textrm{\smallskip  A generalization of this result to several
variables  has been communicated  at Analysis  Belgrade Seminar,
cf. \cite{mm.unpub2}. }

Note that $(i.4)$  is  a corollary of  $(i.2)$.

{\it Outline of proof of} $(i.1)$.  To every  $a\in \partial D$  we
associate a nonnegative harmonic function $u=u_a$. Since  $D$ is
convex, for $a\in
\partial D$, there is a  supporting line $\Lambda_a$ defined by $(w-a,n_a)=0$, where
$n=n_a$  is  a unimodular complex number such that $(w-a,n_a)\geq
0$ for every $w\in \overline{D}$.
%let     $\Lambda_a$ be  the line  defined by $(w-a,n_a)=0$, where
%$n=n_a$  is    is unit inner normal at $a$ with respect to  $\partial D$.
% a unimodular complex number such
% Since  $D$ is convex,   then  $(w-a,n_a)\geq 0$ for every $w\in \overline{D}$.
Define $u(z)=(h(z)-a,n_a)$  and  $d_a= d(h(0),\Lambda_a)$.   Then
$u(0)=(h(0)-a,n_a)= d(h(0),\Lambda_a)$ and therefore, by the mean
value theorem,
% -(a,n_a)=
$$ \frac{1}{2\pi}\int_0^{2\pi} u(e^{it})dt= u(0)=d_a=
d(h(0),\Lambda_a).$$ Since  $u=u_a$  is   a nonnegative harmonic
function,  for    $z=re^{i\varphi}\in\mathbb{D}$, we  obtain
$$
u(z)\geq \frac{1-r}{1+r}\, \frac{1}{2\pi}\int_0^{2\pi}
u(e^{it})dt.
$$

Hence   $u(re^{i\varphi})\geq  d_a (1-r)/2$, and therefore
$|h(z)-a| \geq d_a (1-r)/2   \geq (1-r)R_0/2$. Thus $|h(z)-a| \geq
(1-r)R_0/2 $ for every  $ a\in
\partial D$  and therefore  we obtain  $(1)$: $d(h(z),\partial D)
\geq (1-r)R_0/2 $.   \hfill   $\Box$

Note that  if  $D$ is a convex domain,  then  in  general  for
$b\in
\partial D$ there is no inner normal.   However,  there is  a
supporting line $\Lambda_b$ defined by $(w-b,n_b)=0$, where
$n=n_b$  is  a unimodular complex number such that $(w-b,n_b)\geq
0$ for every $w\in \overline{D}$.

 Note that  proof  of theorem   can also be based on   Harnack's
theorem  (see also \cite{rud2},  Lemma 15.3.7)  or Hopf Lemma.\\
% \lambda_f =

We use the notation  $\lambda_f = l_f (z)= |\partial f (z)|   -
|\bar\partial f (z)|$ \, and   $ \Lambda_f (z)= |\partial f (z)|   +
|\bar\partial f (z)|,$   if $\partial f(z)$   and  $\bar\partial
f (z)$ exist.

\begin{thm}[\cite{revroum01}]\label{t.in.belowA}
$(ii.1)$  Suppose  that   $\,h\,= f + \overline{g}$ is    a Euclidean
orientation preserving  harmonic mapping from $\mathbb{D}$ onto
bounded convex  domain $D=h(\mathbb{D})$, which contains a disc
$\,B(h(0);R_0)\,$. Then $ |f'| \geq R_0/4$ on
$\mathbb{D}$.\\
$(ii.2)$   Suppose, in addition,  that  $h$  is qc.  Then  $l_h\geq (1-k)
|f'|\geq (1-k)R_0/4$    on  \,$\mathbb{D}$   \\
$(ii.3)$  In particular,    $h^{-1}$ is Lipschitz.
\end{thm}
A proof of the theorem can be based on   Theorem  \ref{t1} and \\
$\rm{(b)}$: the approximation of a convex domain with smooth
convex
domains,\\
which is based on  {\it the hereditary property of convex
functions: if an analytic function maps the unit disk univalently
onto  a convex domain, then it also maps each concentric  subdisk
onto a convex domain}. Now we outline an {\it approximation
argument for convex domain $G$}. Let $\phi$ be conformal mapping
of $\mathbb{D}$ onto $G$, $\phi^\prime(0)>0$, $G_n= \phi (r_n \mathbb{D})$,
$r_n=\frac{n}{n+1}$, $D_n= h^{-1}(G_n)$; and $\varphi_n$ conformal
mapping of $\mathbb{D}$ onto $D_n$, $\varphi_n (0)=0$,
$\varphi_n^\prime (0)>0$     and $h_n=h\circ \varphi_n$.  Since  $
D_n\subset D_{n+1}$  and $\cup D_n= \mathbb{D}$, we can apply the
Carath\'{e}odory theorem; $\varphi_n$ tends to $z$, uniformly on
compacts, whence $\varphi_n^\prime(z) \rightarrow 1$
($n\rightarrow \infty$). By hereditary property $G_n$ is convex.

Since the boundary of  $D_n$  is an analytic Jordan curve, the
mapping  $\varphi_n $  can be continued  analytically across
$\mathbb{T}$, which implies that  $h_n$ has a harmonic extension
across  $\mathbb{T}$.  An application of Theorem  \ref{t1} $(i.4)$  to
$h_n$   gives the proof.\\
% and using  XX

\begin{exa}
$\rm{(ii.4)}$  \textrm{\textit{$f(z)=(z-1)^{2}$ is univalent on $\mathbb{D}$. Since $%
f^{\prime }(z)=2(z-1)$ it follows that $f^{\prime }(z)$ tends $0$ if $z$
tends $1$. This example shows that we can not drop the hypothesis that $f(%
\mathbb{D})$ is a convex domain in Theorem  \ref{t1} $(i.4)$. }}\\
%Proposition \ref{prop3}. }}
$\rm{(ii.5)}$  \textrm{\textit{$g(z)=\sqrt{z+1}$ is univalent on
$\mathbb{D}$  is not Lipshitz on $\mathbb{D}$.}}\\
$\rm{(ii.6)}$ $g^{-1}(w)=w^2-1$  and  $(g^{-1})'(w)=2w$  tends to
$0$  if $w\in g(\mathbb{D})$ tends to $0$.
\end{exa}

\subsection{}\textrm{\textit{Hall, see \cite{dur} p. 66-68, proved the
following: }}

\begin{lem}[Hall lemma]
$\rm{(ii.7)}$ \textrm{\textit{For all harmonic univalent mappings $f$ of the
unit disk onto itself with $f(0)=0$, \newline
$|a_{1}|^{2}+|b_{1}|^{2}\geq c_{0}=\frac{27}{4\,\pi ^{2}}$,
\newline where $a_{1}=Df(0)$, $b_{1}=\overline{D}f(0)$ and
$c_{0}=\frac{27}{4\,\pi ^{2}}=0.\ 6839.\ .\ $. . }}\newline
$\rm{(ii.8)}$  If in addition  $f$  is  orientation preserving,  then  $|a_{1}|\geq \sigma_0$,  where  $\sigma_0= \frac{3  \sqrt{3}}{2 \sqrt{2}\,\pi}$.
\end{lem}

Hence, one can derive:
%\begin{thm}

$\rm{(I0)}$    There is a constant $c>0$  such that
%for every $x\in G$\\(a) $B(x,r)\subset G$
if  $h$ is harmonic planar mapping   and   $B(hx,R)\subset
h(B(x,r))$, then $r \Lambda_h(x) \geq c R$.

Now we give another proof and generalization of the part $(ii.3)$
of  Theorem  \ref{t.in.belowA}, which is not based on the
approximation of a convex domain with smooth convex domains.

\begin{thm}\label{1clip}
Suppose  that ${\rm (a.1)}$:   $D$  and  $D_*$  are simply-connected hyperbolic domains in
$\mathbb{R}^2$ with non-empty boundary  and that   $\,f\,$ is    a euclidean  harmonic   univalent
mapping from $D$ onto
$D_*$. Then  \\
${\rm (iii.1)}$  $d \Lambda_f \succeq d_*$  on $D$, where $w=f(z)$, $d=dist(z, \partial D)$ and  $d_* (w)=dist(w, \partial D_*)$.  \\
${\rm (iii.2)}$  If in addition to hypothesis  ${\rm (a.1)}$ we
suppose that ${\rm (a.2)}$: $D$ is a $C^{1,\alpha}$, $0 < \alpha
<1$, domain and $D_*$ is convex  bounded domain, then there is a
constant $c>0$ such that $\Lambda_f \succeq c$ on $D$.
\end{thm}
By Theorem \ref{oneone}(Kellogg) we can reduce the proof to the
case $D=\mathbb{D}$.
\begin{proof}
Let  $z\in D$  and  $\phi_0$ conformal of   $\mathbb{B}$ onto  $D$
such that   $\phi_0(0)= z$  and $F=h\circ \phi_0$.  Since
$\Lambda_F= \Lambda_f |\phi_0'|$,
 by $\rm{(I0)}$    version of Hall  lemma,  $\Lambda_F\succeq d_*$,  $|\phi_0'|\asymp d$.
Hence  we find   ${\rm (iii.1)}$:  $d \Lambda_f \succeq d_*$.

Now let  $z=\phi(z')$ be  a conformal of   $\mathbb{B}$ onto  $D$
and $d'=dist(z')$.  Then  $d' |\phi'|\asymp d$  and by Theorem
\ref{oneone}(Kellogg)  ${\rm (iii.3)}$:  $d' \asymp d$.
%and  $\phi$ conformal of   $\mathbb{B}$ onto  $D$

Since $D_*$   convex, $d_* \succeq d'$ and therefore by ${\rm
(iii.3)}$  we find $d_* \succeq d$. Hence by ${\rm (iii.1)}$
$\Lambda_f \succeq c$ on $D$.
\end{proof}

The following result is an immediate  corollary of Theorem
\ref{1clip}.

\begin{thm}
If in addition to  hypothesis ${\rm (a.1)}$ and ${\rm (a.2)}$ of
Theorem \ref{1clip} $\,f\,$ is qc, then $f^{-1}$  is Lipschitz  on
$D_*$.
\end{thm}
\begin{cor}
In particular, if $\,f\,$ is conformal, then $f^{-1}$  is Lipschitz
on $D_*$.
\end{cor}
%Thus,
\begin{proof}
 Using that $\,f\,$ is qc, it follows that $\lambda_f \succeq
\Lambda_f \succeq c$ and the rest of the proof is routine.
\end{proof}
%[14].
%%%%%
%  and
%%%%
Let  $G$ be simply connected  hyperbolic planar domain  and  $\rho=\rho_G^{hyp}$  hyperbolic density;  we also  use short  notation  $\rho_{hyp}=\rho_G$.
Using  the uniformization theorem, one can define  hyperbolic density  for a  hyperbolic planar domain.

If   $G$  is  a  planar  domain  with non-empty boundary,  $f$  a  $C^1$   complex valued mapping  $d=d_G$ and  $d_*=d_{f(G)}$,
we define $H:=|\partial f|^{qv}_{hyp}=\frac{d(z)}{d_*(f(z))} |f_z|$.  In addition if $G$ is  hyperbolic planar domain, $\rho=\rho_G$
and $\rho_*=\rho_{f(G)}$,  we define    $H:=|\partial f|_{hyp}=\frac{\rho_*(f(z))}{\rho(z))} |f_z|$.

\begin{thm} If  $G$ is  simply connected, then\newline
${\rm (iv.1)}$    $d_G (w) \leq  \rho_G^{-1} \leq 8 d_G (w)$.\newline
There is an absolute  constant $c$   such that
under the hypothesis  ${\rm (a.1)}$ of Theorem \ref {1clip}  if $f$  in addition  orientation preserving, we have   \newline
${\rm (iv.2)}$ $H:=|\partial f|^{qv}_{hyp}=\frac{d(z)}{d_*(f(z))} |f_z| \geq c $.\newline
%In addition if  $G$ is  simply connected, then\newline
${\rm (iv.3)}$      $H:=|\partial f|_{hyp}\geq c/8$.
\end{thm}
%By Hall lemma, $\rm{(ii.8)}$,  $|f_z|  ||\geq sigma_0$
\begin{proof}
Let   $\phi : \mathbb{D} \rightarrow G$ be   conformal.
Then
$\rho(w) |\phi'(z)| = \rho(z)$  and $d_G (w) \leq d_\mathbb{D }(z)|\phi'(z)|\leq 4 d_G (w) $.  Hence  $d_G (w) \leq  \rho_G^{-1} \leq 8 d_G (w)$.

Set   $w= \phi(z)$,    $\zeta=f(w) $ and  $F=f\circ \phi$.   Since  $F_z =f_w \phi'(z)$,    by an application of Hall lemma, $\rm{(ii.8)}$, to  $F$,  we find   $|f_w(w)|  |\phi'(z)|\geq d_* \sigma_0$.  If $w\in G$, we can choose  a conformal mapping  $\phi_0: \mathbb{D} \rightarrow G$  such that  $w= \phi_0(0)$. Then  $|\phi_0'(0)|\leq 4d$ and    therefore  $4d |f_w(w)|\geq d_* \sigma_0$. Hence we  get  ${\rm (iv.2)}$  with $c= \sigma_0/4$.
%  $|\phi_0'(0)|=\rho(w)^{-1} $;
\end{proof}
%%%%
The part ${\rm (iv.3)}$  has also been proved by Kalaj.

\subsection{The minimum principle for the Jacobian}
(I1)  Let $h :\Omega
 \rightarrow \mathbb{C}$   be a harmonic map whose Jacobian
determinant   $J = |h_z|^2 -  |h_{\overline{z}}|^2 $  is positive
everywhere in $\Omega$. Then $-\ln J$ is subharmonic; More
precisely, cf.  \cite{Iw14,IwOn},

$$-\frac{1}{4} \Delta \ln J= -(\ln J)_{z\overline{z}} = \frac{|h_{zz}\overline{h_{\overline{z}}} - \overline{h_{\overline{z}\,\overline{z}}}   h_z|^2}{J^2}\,.$$
% used interestingthat fact
Note that in \cite{man}  it is  proved previously  that\\
(I2)    $X= log
\frac{1}{J_h}$ is a subharmonic function.

We left to the reader to check  the following fact (I3-I6):

If $F$ is an  analytic  function, then   $|F|^2_{z\overline{z}}= |F'|^2$.  Hence

(I3)  $J(h)_{z}=f'' \overline{f'} -g'' \overline{g'}   $,  $J(h)_{z\overline{z}}= |f''|^2- |g''|^2 $.

(I4) In general,   $J(h)$ nether subharmonic nor   superharmonic.

If   $\tau:D\rightarrow I $, $I=(a,b)$,  $\chi:I \rightarrow R$, then

(I5)   $(\chi\circ\tau)  _{z\overline{z}}= (\chi''\circ\tau)  \tau_{\overline{z}} \tau_{z} + (\chi'\circ\tau) \tau_{z\overline{z}} =(\chi''\circ\tau)|\tau_{z}|^2 + (\chi'\circ\tau) \tau_{z\overline{z}} $.

If we set  $\chi(x)=x^{-1}$  and  $\tau=J $, we find

(I6) $-(J^{-1})_{z\overline{z}} J^3=|\tau_{z}|^2 + |B+C|^2$.

 If we set  $\chi=  log$,   $\tau=J $,   $B= f' g''$  and   $C= g' f''$, then

(I7) $-(\ln J)_{z\overline{z}} J^2=|B+C|^2$.

(I8) Suppose that    $F$ and $H$   are analytic function  in  a domain $G$ such that
$|F|^2
= m_0 + |H|^2$ on $G$, where $m_0$ is a positive constant.  Then    $F_z \overline{F}= H_z \overline{H}$  and therefore

$H'/F'= \overline{F}/\overline{H}$. Hence $H'/F'=a_0$, where $a_0$ is a  constant  and therefore  $H/F=a_0 z +a_1$.

$|F|^2
= m_0 + |a_0 z +a_1|^2  |F|^2$ on $G$.  Without loss of generality we can suppose that  $0 \in G$ and that  $F=b_0 +b_n z^n + o(z^n)$, $n\geq 1$.

This leads to a contradiction and so  $F=b_0$ and therefore $H=c_0$ on $G$, where  $b_0$ and $c_0$ are   constants.

\begin{prop}
$\mathrm{(a.2)}$  If $h$ is harmonic on $U$ and  $J(h)$  attains minimum
(different from $0$) at interior point $a$, then  $J(h)$  is
constant  function, and  $h$ is affine.

$\mathrm{(b.2)}$     Functions  $log \frac{1}{J_h}$ and $\frac{1}{J_h}$ are
subharmonic function.

$\mathrm{(c.2)}$ In particular,  $log J_h$  is a  superharmonic function.
\end{prop}

\begin{cor}[Minimum Principle] Let $h :\Omega
 \rightarrow \mathbb{C}$  be a harmonic
map whose Jacobian determinant $J$  is positive everywhere in
$\Omega$. Then   $\inf_ F J \geq inf _{\partial F} J$ for every
compact $F\subset \Omega$.
\end{cor}

Example   $h=f + \overline{g}$,  where  $f= 4 z$,  $g=z^2/2$,
shows that that analog statement is not valid for maximum;  $J= 16
- (x^2 + y^2)$ attains maximum  $16$ at $(0,0)$.

In general,   $min J$    is attained at the boundary.

Proof of (a.2). Suppose that $h$ is orientation preserving  and
$J(h)= |f'|^2 - |g'|^2$ attains min (different from $0$)  at
interior point $a$; then $|f'|^2 - |g'|^2\geq J(h,a)=m$,   $|f'|^2
\geq m + |g'|^2$ and therefore   $1 \geq s(z)$, where
$$s=\frac{m}{|f'|^2 }  +  \frac{|g'|^2}{|f'|^2 }\,.$$
Since $s$ is subharmonic  and  $s(a)=1$,  $s$  is a constant, ie.  $s=1$.  Hence $|f'|^2
= m + |g'|^2$, ie. $J=m$.  By   (I8),   $f'$  and  $g'$ are constant  functions  and therefore   $h$ is affine.

Proof of (b.2).   Hence,  since  $\exp$ is a convex increasing function,  it follows that
 $\exp\circ X=\frac{1}{J_h}$  is   also a subharmonic function.

%  but  we do not use it  here. ?? (not increasing ??)

Although   $\chi(x)= e^{-x}$ is convex
%and therefore not  (wrong !)
the conclusion that  $\chi \circ X=J(h)$  is   a subharmonic
function is not true  in general. Note that here $\chi$   is a
decreasing function.
In general, the   minimum  modulus  principle for complex-valued
harmonic functions is not valid; see the following examples:

\begin{exa}
1. If $f(z)= x + i$, then $|f(z)|^2= x^2 +1$   and  $|f|$  attains
minimum which is $1$  for every points on
$y$ axis\\
2. If  $f_c(z)= x + i(x^2-y^2 +c)$, then  $J_f=-2y$.  Let
$d:\mathbb{C} \rightarrow R$  is given by $d(z)= |z|$, $z \in
\mathbb{C}$, $g=f_{-1}$, $C(x)= x + i(x^2 - 1)$ and $D=\{(x,y): y
< x^2 -1\}$. Since $g(\mathbb{C})= D$  and $0\notin D$, then $d$
attains minimum on $tr(C)$ at some point $w_0$ and there is a real
point $x_0$   such that $g(x_0)=w_0$, $g$ maps $\mathbb{C}$ onto
$D$ and $|g|$ attains minimum at $x_0$.

Let  $c < 0$  and  $D_c=\{(x,y): y < x^2 -|c|\}$. Then
$f_c(\mathbb{C})= D_c$  and $0\notin D_c$.
% Whether
At   first sight someone can guess    that  $|f_c(0)|= |c|$  is
the minimum value for $|f_c|$ if $c < 0$?  We leave to the
interested reader to show that $|c|$  is not the minimum value.
\end{exa}
%%%%%
%THEOREM
\subsection{Outline of proof of Theorem RCK  given in \cite{Iw14,IwOn}} For $f: \mathbb{T}\rightarrow \mathbb{C}$, we define $\b{f}$
on $[0, 2\pi]$ by $\b{f}(t)= f(e^{it})$.
Let $\gamma$ be a closed Jordan curve,  $G=\rm{Int}(\gamma)$,  $f_0 : S^1 \overset{\text{onto}}{\longrightarrow}
\gamma$ a monotone map  and $F=P[f_0]$.

\begin{thm}[T. Rad\'{o} H. Kneser  G. Choquet, Theorem RKC] If $G$  is convex, then $F$  is a
homeomorphism of $\mathbb{D}$  onto $G$.
\end{thm}

Iwaniec-Onninen  \cite{Iw14,IwOn}  presented  a new analytic proof
of RKC-Theorem.
%The novelty lies in establishing
 The approach is based on the following steps.\\
$\rm{c1})$ Prove the theorem if $f_0$ is  diffeomorphism  and $G$
is  a smooth strictly convex domain,  using the minimum principle
for the Jacobian determinant
%An advantage and usefulness of this proof is that it provides
and explicit interior lower bounds on the Jacobian in terms of the
regularity of the domains and the boundary map.\\
%In the first step toward full generality of the RKC-theorem we
% relax regularity of the boundary map.
$\rm{c2})$   Let $G$  be still a smooth strictly convex domain,
but $f : S^1 \overset{\text{onto}}{\longrightarrow} \gamma$ an
arbitrary monotone map. This map can easily be shown to be a
uniform limit of diffeomorphisms. Now the Poisson extensions $F_j$
are harmonic diffeomorphisms  in $\mathbb{D}$, converging
uniformly on $\overline{\mathbb{D}}$ to $F$  and  $J_F > 0$, cf
also \cite{kal.studia,mm.unpub2}.\\
%We claim that $J_F > 0$.
$\rm{c3})$  the approximation of a convex domain with smooth
convex domains.\\
$\rm{c4})$   there is a conformal map  $\phi$ of $\mathbb{D}$  onto $G$; by  variation of boundary values deform
this conformal map into a   harmonic diffeomorphism.

It is convenient  to give  kinematic description of  $f : \mathbb{S}^1
\overset{\text{onto}}{\longrightarrow} \gamma$,  to view  it as motion of an
object along $\gamma$  in which $\mathbb{S}^1$  is labeled as a clock.
As time runs from $0$ to $2\pi$  the motion  $t\rightarrow
 \b{f}(t)$  begins at the point  $\b{f}(0)=0$  and terminates at the same
 point $\b{f}(2\pi)=0$.  The velocity
vector  $\upsilon(t)= \b{f}'(t)$  is tangent to $\gamma$   at
$s=s(t)=\int_0^t |\upsilon (\tau)| d \tau $. We call   $|\upsilon
(t)|$  the speed.
Let $z = z(s)$   be the length parametrization of $\gamma$.  Since
$|z'(s)|=1$,   we have  $z(s)= e^{i \varphi(s)}$, where
$\varphi(s)$   referred to as the tangential angle, is uniquely
determined by the arc parameter $s$  because $G$  is smooth and
strictly convex.  The derivative is exactly the curvature of
$\gamma$;  that is  $\kappa(s)=\varphi'(s)$.
The speed   $|\upsilon (t)|$, being positive, uniquely represents
unique diffeomorphism $f : S^1
\overset{\text{onto}}{\longrightarrow} \gamma$.   An explicit
formula for $f$ involves the curvature of   $\gamma$.

\begin{thm}[\cite{Iw14,IwOn}] \label{Iwt1} Let  $f : S^1 \overset{\text{onto}}{\longrightarrow}
\gamma$   be a  $C^{\infty}$ -difeomorphism and $F: \mathbb{D}
\rightarrow C$ its continuous harmonic extension. Then $F :
\mathbb{D} \overset{\text{onto}}{\longrightarrow}  G$ is a
$C^1$-smooth diffeomorphism whose Jacobian determinant satisfies:

\begin{equation}\label{Iwjac1}
J_F\geq \frac{k m^3}{2\pi KM}  \, \, {\rm  everywhere\,   in}\, \,   \mathbb{D}
\end{equation}
provided  $0<k \leq  \min  \kappa(s)\leq  \max  \kappa(s)  \leq
K$   and  $0< m   \leq  \min  \upsilon (t) \leq  \max  \upsilon
(t)(s) \leq M$, where  $\upsilon (t)=|\underline{f}'(t)|$.
\end{thm}

In \cite{Iw14,IwOn} the  strategy  is used  to prove  first The
Lower Bound of the Jacobian along $\mathbb{S}^1$.
%We will eventually be reduced
Then the proof  is reduced
to showing inequality  (\ref{Iwjac1})
at the boundary of the disk, as one may have expected from the
Minimum Principle.
%+_+)}{:@?"£$^&<>!£%&*()_+~

\subsection{}We also can use  Theorem \ref{tjac1} below, which yields appriori estimate, instead of Theorem \ref{Iwt1}  in the procedure of proof of Theorem RKC.

Let  $ A(\gamma)$   be the family of  $C^1$ -difeomorphism $f :
\mathbb{S}^1 \overset{\text{onto}}{\longrightarrow} \gamma$.   Set
$\upsilon (t)=\upsilon_f (t)=|\underline{f}'(t)|$.

Let   $d$ be diameter   of $G$, $b\in G$,  $B=B(b;d/2$, $\gamma_1$
the part of $\gamma$  out  of $B$  and $I=\{ t : f(t)\in
tr(\gamma_1)\}$.  Then   $d/2  \leq  |\gamma_1| \leq M |I| $.

\begin{thm}\label{tjac1}
Let  $f\in  A(\gamma)$    and  $0< m \leq  \min  \upsilon (t) \leq  \max  \upsilon (t)(s)
\leq M$.  Then
\begin{equation}\label{jac1}
J_F\geq \frac{d m}{8  \pi M}\, \, {\rm  on} \, \, \mathbb{D}.
\end{equation}
\end{thm}

Using approuch outlined in \cite{napoc1,mm.fil12a}, we can prove

\begin{lem}
The inequality {\rm (\ref{jac1})}  holds everywhere in
$\mathbb{S}^1$.
\end{lem}

Naturally, the interior estimate at {\rm (\ref{jac1})}   would
follow from the already established estimate at the boundary (via
the minimum principle) if we knew that the Jacobian of $F$   was
positive in $\mathbb{D}$.

But, by  Theorem RKC  the Jacobian of
$F$  is  positive in $\mathbb{D}$.
Theorem \ref{kal.publ2.8} below   yields  better estimate.
%the  approximation
%argument for convex domain,  $\rm{c2})$
Using
Theorem \ref{t1}, the part  $(i.3)$,  and the minimum principle for Jacobian  one can derive:

\begin{thm}\label{kal.publ2.8}
$\rm{d1})$   Let $\Omega$ be  a  convex Jordan domain,     $f:\mathbb{T}\rightarrow \partial\Omega$ absolutely continuous homeomorphism which preserves orientation,  and let  $w = P[f]$  be a harmonic function between the unit disk and
$\Omega$, such that  $w(0) = 0$,\\
$\rm{d2})$   $|\underline{f}'(t)|\ge m$,  for almost every  $0\leq t \leq 2\pi$,\\
$\rm{d3})$   $\underline{f}'$ is Dini's continuous. \newline
Then the following
results  hold  ${\rm (iv.1):}$ \,   $J_w(z) > m\, dist(0,\partial\Omega)/2 $,
for every $z \in \mathbb{D},$  and \newline
${\rm (iv.2)}$    $w$ is bi-Lipschitz.
\end{thm}
Let $X$  be a compact subset of a metric space with metric $d_1$ (such as $\mathbb{R}^n$) and   let $f:X\rightarrow Y$ be a function from $X$  into another metric space $Y$  with metric $d_2$ . The modulus of continuity of $f$ is
$\omega_f(t) = \sup_{d_1(x,y)\le t} d_2(f(x),f(y)) \,$, $t>0$.
The function $f$ is called Dini-continuous if
$$\int_0^1 \frac{\omega_f(t)}{t}\,dt < \infty. $$   Dini continuity is a refinement of continuity. Every Dini continuous function is continuous. Every Lipschitz continuous function is Dini continuous.
Note that under the above hypothesis  $\underline{f}'$  has   continuous extension to  $[0,2\pi]$  and  partial derivatives of  $w$ have  continuous extension to
$\overline{\mathbb{D}}$  and one can show that  $w$ is bi-Lipschitz.
%It seems that
We can use an aproximation argument to  prove ${\rm (iv.1)}$  for  $C^{1,\alpha}$ domains  $\Omega$  without the hypothesis
$\rm{d3})$.  Moreover,    an application of $H^p$  theory shows  that  the following result, due  to  Kalaj, holds  in general:
% only
\begin{thm}[Theorem 2.8,  Corollary 2.9 \cite{kal.publ}]\label{kal.publ2.8b}
Under hypothesis  $\rm{d1})$  and  $\rm{d2})$ of  Theorem  \ref{kal.publ2.8},  $J_w^*$ exsists  a.e. on $\mathbb{T}$  and  ${\rm (iv.1)}$    holds.
\end{thm}
% Iwaniec has communicated the proof of  Rado-Kneser-Choquet  theorem. the
Under hypothesis  $\rm{d1})$  of  Theorem  \ref{kal.publ2.8},  Theorem RKC  states  that \newline
${\rm (v.1)}$   $J_w >0$  on  $\mathbb{D}$. \newline
{\it Question}\,1.   Can we modify  approach in \cite{Iw14,IwOn}  to  give  analytic
proof  of  Theorem  \ref{kal.publ2.8b} (of course without appeal to Theorem RKC (moreprecisely  to ${\rm (v.1)}$)?
\section{Estimates for the modulus of the
derivatives of harmonic univalent mappings in space}\label{s.space}

% A function $u$ is called superharmonic if  $-u$ is  subharmonic.

\begin{dfn} Let $G$   be a subset of the Euclidean space
${\mathbb{R}}^n$ and let
$$\varphi \colon G \to {\mathbb{R}} \cup \{+ \infty \} $$
be an lower  semi-continuous function.  Then,  $\varphi$   is
called $q$- superharmonic, $0< q \leq 1$,   if for any closed ball
$\overline{B(x,r)}$ of center $x$ and radius $r$ contained in $G$
and every real-valued continuous function $h$ on
$\overline{B(x,r)}$ that is harmonic in  $B(x,r)$   and satisfies
$\varphi(y) \geq q h(y)$ for all $y$   on the boundary  $\partial
B(x,r)$  of $B(x,r)$ we have $\varphi(y) \geq h(y)$   for all   $y
\in B(x,r)$. If $q=1$  we  say superharmonic  instead of  $1$-
superharmonic. \end{dfn}

\begin{dfn}\label{dfn.wH} Let  $G$ be a  domain  in $\mathbb{R}^n$.   Suppose that
$f:G\rightarrow \mathbb{R}^n$  is a $C^1$ function and   there is
a constant $c>0$  such that
for every $x\in G$\\
(a)  if $B(x,r)\subset G$   and   $B(fx,R)\subset f(B(x,r))$, then
$r \Lambda_f(x) \geq c R $.\\
(a') if  $B(fx,R)\subset f(\underline{B}(x,r))$, then
$r \Lambda_f(x) \geq c R $.

We say that $f$  has $H$-property  (respectively weak $H$-property)   if  (a) (respectively (a'))    holds.
\end{dfn}
 By Hall  lemma  planar euclidean harmonic mappings  have   $H$-property.

We say that   $\underline{F}\subset HQC_K(G,G')$   has $H$-property   if $F$  is closed  with respect to uniform convergence, and for $f\in  \underline{F}$,
$J(f)$  has no zeros  in $G$.

If  $\underline{F}\subset HQC_K(G,G')$   has $H$-property, then   $f\in\underline{F}$  has   weak $H$-property.

By Lemma  \ref{lhall1}, there is  a constant $c>0$  such that $d(x) \Lambda_f(x)\geq c d(fx)$,  $x \in G$.
\begin{lem}\label{lhall1}
Let $F_K$  be  a family of   harmonic  K-qc mapping $f:\mathbb{B
}\rightarrow R^n$ such  that for $f\in  F_K$, $J(f)$  has no
zeros,   $F_K$  is closed  with respect to uniform convergence,
$f(\mathbb{B})\supset \mathbb{B}$ and  $f(0)=0$. Then  there is a
constant $c>0$  such that if $f\in  F_K$   is harmonic  K-qc
mapping $f(\mathbb{B})\supset \mathbb{B}$, $f(0)=0$, then
$\Lambda_f(0) \geq c$.
\end{lem} Contrary  there is
a  sequence $f_n\in  F_K$   such that  $\Lambda_{f_n}(0)\rightarrow 0$.  Sequence $f_n$    forms a normal family and there is
a subsequence of $f_n$ which  converges  uniformly  to a limit     $f_0\in  F_K$; this is a contradiction.

Suppose that $\,F\,$ is
%a euclidean harmonic
mapping from the unit ball $\mathbb{B}\subset \mathbb{R}^n$ into
$\mathbb{R}^m$ and
% and if  $F'_r(x)$ exists,
suppose  that $\omega=F^*(x)$ and $(\partial_r F)^*(x)$ exist at a
point $x\in
\mathbb{S}$.    Then \\
$(A1)$  $F'_r(x)$ exists and $(\partial_r F)^*(x)=F'_r(x)$.

{\it Proof of} $(A1)$.  By Lagrange theorem, there is  $t_k \in
[r,1)$ such that $F_k(rx)- F_k(x)= - (F_k)'_r(t_k x) (1-r)$.
Hence, since $(\partial_r h)^*(x)$ exist at a point $x\in
\mathbb{S}$, if $r$ tends $1$,  then  $(F_k)'_r(t_k x)$  tends
$(F_k)'_r( x)$.

\begin{thm}\label{thm.space1}
 \rm {(a1)}  Suppose  that   $\,h\,$ is    a euclidean harmonic mapping
from the unit ball $\mathbb{B}\subset \mathbb{R}^n$ onto  a
bounded convex  domain  $D=h(\mathbb{B})$, which
contains the  ball  $\,B(h(0);R_0)\,$.  Then \\
$\rm{(i.1)}$  $d(h(z),\partial D) \geq (1-|z|)\overline{c}_n R_0
$, \, $z\in \mathbb{B}$,  where $\overline{c}_n=
\frac{1}{2^{n-1}}$.\\
$(i.2)$    For every  $x \in \mathbb{S}$  and for  $0<r <1$, there
is $t \in
[r,1)$  such that  $|h'_r(t x)|\geq c_0$.\\
%e^{i\theta}
%If  $a=h^*(e^{i\theta})$ and  $h^*_r=h'_r(e^{i\theta})$ exist,
%,where    $c_0=\frac{R_0}{2}$.
$(i.3)$  If  $h$ is K-qc   and   $(a2)$   $h$  has    $H$-property   or $ \rm {(a3)}$  for some  $k$,   $|h'_{x_k}|^n$  is  $q$- super harmonic,
 then $h$ is  co-Lipschitz on  $\mathbb{B}$.

 \rm {(a4)}  Suppose, in addition, to \rm {(a1)}    that  $h$ is K-qc   and  $D$ is $C^2$ domain.  Then \\
 $(i.4)$     $h'_r(x)$ exists and $(\partial_r h)^*(x)=h'_r(x) \geq c_0$  for almost everywhere  $x \in S$.
\end{thm}
%$(ii.1)$
%In addition  to the hypothesis  (a),

{\it Proof of} $(i.1)$.  To every  $a\in \partial D$  we associate
a nonnegative harmonic function $u=u_a$. Since  $D$ is convex, for
$a\in
\partial D$, there is a  supporting hyper-plane $\Lambda_a$ defined by $(w-a,n_a)=0$ where
$n=n_a\in T_a\mathbb{R} ^n$  is  a  unit vector such that
$(w-a,n_a)\geq 0$ for every $w\in \overline{D}$.
%let     $\Lambda_a$ be  the line  defined by $(w-a,n_a)=0$, where
%$n=n_a$  is    is unit inner normal at $a$ with respect to  $\partial D$.
% a unimodular complex number such
% Since  $D$ is convex,   then  $(w-a,n_a)\geq 0$ for every $w\in \overline{D}$.
Define $u(z)=(h(z)-a,n_a)$  and  $d_a= d(h(0),\Lambda_a)$.  Then
$u(0)=(h(0)-a,n_a)= d(h(0),\Lambda_a)$.  Let  $a_0 \in \Lambda_a$
be the point such that  $d_a= |h(0)- a_0|$.  Then   from geometric
interpretation it is clear that $d_a \geq R_0$.

By  Harnack's inequality, $\overline{c}_n (1-r)  u(0)  \le u(x)$,
$x \in \mathbb{B}$ and  $r=|x|$,  where $\overline{c}_n=
\frac{1}{2^{n-1}}$.  In particular,  $\overline{c}_n d(x)  R_0 \le
u(x)\le |h(x) -a|$   for every  $a \in \partial D$. Hence, for a
fixed $x$,  $d_h(x)=\inf_{a \in \partial D}|h(x) -a| \geq
\overline{c}_n R_0  d(x)$  and therefore  we obtain  $(i.1)$.
\hfill   $\Box$

{\it Proof of} $(i.2)$.  Set $u(z)=(h(z)-h(x),n_{h(x)})$ and
$c_0=\overline{c}_n R_0 $. Since $u(x)=0$, then \\
(i.5) for $x \in \mathbb{S}$,   $u(rx)- u(x)\geq c_0 (1-r)$.

By Lagrange theorem, there is  $t \in [r,1)$ such that $u(rx)-
u(x)= u'_r(tx) (1-r)$.
\hfill   $\Box$

{\it Proof of} $(i.3)$. Suppose for example $(b1)$. By Theorems  \ref{thm.space1}  and  \ref{aaaM1}, $ \underline{J}(x)\succeq \frac{d_*}{d}  \succeq c$  for every  $x\in \mathbb{B}$.
By $ \rm {(a3)}$,

$|h'_{x_k}(x)|^n \geq q \frac{1}{|\underline{B}_{x}|}\int_{\underline{B}_{x}}|h'_{x_k}(z)|^n dz\,,x\in \mathbb{B}$.  Hence   $|h'_{x_k}(x)|\succeq \underline{J}(x)\succeq c$

and  $\lambda(h'(x)\succeq c$  and therefore  $h$ is  co-Lipschitz on  $\mathbb{B}$.\hfill   $\Box$

{\it Proof of} $(i.4)$.
%First,  in addition  to the hypothesis  (a1),   suppose  that \rm {(a5)}:
%there exists the unit  inner normal $n=n_{\omega}$ at $\omega=h^*(x)$,  $x\in \mathbb{S}$, with respect to $\partial D$  and
First,  suppose  that   $h'_r(x)$ exists for some $x \in
\mathbb{S}$. By $(i.5)$, $(\frac{h(rx)-h^*(x)}{1-r},
n)\geq c_0 $,  where  $n=n_a$  and  $a=h^*(x)$.  Hence, since  $h'_r(x)$ exists, it follows  \\
%Then \\  $(i.6)$
$(i.6)$   $(h'_r(x), n)\geq \overline{c}_n R_0$.\\
% Outline of
%  and there is a point
%{\it Proof of} $(i.6)$.

By a result of D. Kalaj  \cite{k.jdam}, partial derivatives of $h$ are bounded and therefore    $h'_r(x)$ exists for almost everywhere  $x \in S$
and therefore   by  $(i.2)$ and $(i.6)$, $(\partial_r h)^*(x)=h'_r(x) \geq c_0$  for almost everywhere  $x \in \mathbb{S}$.
\hfill   $\Box$

In \cite{KaMpacific}  it  is proved the following theorem: a $K$
quasiconformal harmonic mapping of the unit ball $\mathbb{B}^n$
($n>2$) onto itself is Euclidean bi-lipschitz, providing that
$u(0) = 0$ and that $K<2^{n-1}$, where $n$ is the dimension of the
space. It is an extension of a similar result for hyperbolic
harmonic mappings with respect to hyperbolic metric (see Tam and
Wan, (1998)). The proof makes use of M\"obius transformations in
the space, and of a recent result which states that, harmonic
quasiconformal self-mappings of the unit ball are Lipschitz
continuous; this result first has been proved by  the first author
and then generalized also by the second  author.\newline
%%%%%
%%%%%  , mentioned in the introduction,
Introduce the quantity
%\cite{ast.ge}
\begin{equation*}
a_{f}(x)=a_{f,G}(x):=\mathrm{exp}\left( \frac{1}{n|B_{x}|}\int_{B_{x}}{\text{%
log}}J_{f}(z)dz\right), \, \, x\in G,
\end{equation*}%
associated with a quasiconformal mapping $f:G\rightarrow f(G)\subset \mathbb{%
R}^{n}$; here $J_{f}$ is the Jacobian of $f$; while $\mathbf{B}_{x}=\mathbf{B%
}_{x,G}$ stands for the ball $B(x;d(x,\ \partial G))$; and $|\mathbf{B}_{x}|$
for its volume.
Astala and Gehring \cite{ast.ge} observed that for certain distortion
property of quasiconformal mappings the function $a_{f}$,
defined the above,
%in section %\ref{s2},
plays analogous role as $|f^{\prime }|$ when $n=2$ and $f$ is
conformal; and %In \cite{ast.ge},
they establish quasiconformal version of the well-know result due to Koebe,
cited here as Lemma \ref{lem.AG0}:

%Lemma 3.
\begin{lem}{\cite{ast.ge}}\label{lem.AG0} Suppose that G and $G^{\prime }$ are domains
in $R^{n}$: If $f:G\rightarrow G^{\prime }$ is K-quasiconformal, then

\begin{equation*}
\frac{1}{c}\frac{d(f(x),\partial G^{\prime })}{d(x,\partial G)}\leq
a_{f,G}(x)\leq c\frac{d(f(x),\partial G^{\prime })}{d(x,\partial G)},\quad
x\in G,
\end{equation*}%
where $c$ is a constant which depends only on K and n.
\end{lem}

%%%%
%%%%%
Our next result concerns the quantity
%where
\begin{equation*}
\underline{E}_{f,G}(x):=\frac{1}{|\underline{B}_{x}|}\int_{\underline{B}_{x}}J_{f}(z) dV(z)\,,x\in G,
\end{equation*}%
associated with a quasiconformal mapping $f:G\rightarrow f(G)\subset \mathbb{%
R}^{n}$; here   $dV(z)=dz$ is the Euclidean volume element  $dz_1 dz_2\cdots dz_n$  and $z=(z_1\cdots z_n)$  and  $J_{f}$ is the Jacobian of $f$; while $\underline{B}_{x}=%
\underline{B}_{x,G}$ stands for the ball $B(x,\ d(x,\ \partial G)/2)$ and $%
|B_{x}|$ for its volume.

Define
%$\alpha_0= \alpha^n_K=  K^{1/(1-n)}$,  \, $C_*=\theta_K^n(1/2)$, $\theta_K^n(c_*)=1/2$   and

\begin{equation*}
\underline{J}_{f}=\underline{J}_{f,G}=\sqrt[n]{\underline{E}_{f,G}}\,.
\end{equation*}

\begin{thm}[\cite{aaa}] \label{aaaM1}
\label{GA1} Suppose that $G$ and $G^{\prime }$ are domains in $\mathbb{R}^{n}$: If $%
f:G\rightarrow G^{\prime }$ is K-quasiconformal, then

\begin{equation*}
\frac{1}{c}\frac{d(f(x),\partial G^{\prime })}{d(x,\partial G)}\leq
\underline{J}_{f,G}(x)\leq c\frac{d(f(x),\partial G^{\prime })}{d(x,\partial
G)},\quad x\in G,
\end{equation*}%
where $c$ is a constant which depends only on $K$ and $n$.
\end{thm}

If  $G$ and $G'$ are domains in $ \mathbb{R}^n$,  by $QCH(G,G')$
(respectively  $QCH_K(G,G')$ ) we denote the set of Euclidean
harmonic quasiconformal mappings (respectively K-qc)  of $G$ onto
$G'$.

If  $D$ is a domain in
$ \mathbb{R}^n$,  by $QCH(D)$ we denote the set of Euclidean
harmonic quasiconformal mappings of $D$ onto itself.
\begin{dfn}
Let $Q_H (G)$   be a family of    harmonic mappings $h$  from $G$
into $\mathbb{R}^n$ such that\\
${\rm (b1)}$   $J(h)$ has no zeros in  $G$,  and\\
${\rm (b2)}$   which is closed with respect to uniform
convergence on compact subsets  and\\
${\rm (b3)}$  for every  sequence  $x_n$  which tends to  a $x_0
\in \partial G$,  $h_n\in Q_H$,   where  $h_n (x)= \frac{1}{d_n} h
((d_n x +x_n))$  and  $d_n=d(x_n)=dist(x_n)$. If $G$ is the unit
ball  $\mathbb{B}$ we write $Q_H $ instead of  $Q_H (G)$.
%the unit ball $\mathbb{B}$
\end{dfn}
Using a criteria  for normality of  a qc family, one  can
establish  criteria when  a  subfamily $Q$ of $QCH(G)$,  for which
$K_O(f) < 3^{n-1}$ for  every $f\in Q$,  is $Q_H (G)$-family.

\begin{dfn}\label{dfn3.1}
Let   $f:G \rightarrow G'$ be a $C^1$ function. We say that   $f$  has Jacobian non zero normal family-property  if\\
${\rm (b1)}$   $J(f)$ has no zeros in  $G$,  and\\
${\rm (b2)}$  for every  sequence  $x_n$  which tends to  a $x_0
\in \partial G$,  $(f_n)$  forms a normal family,   where  $f_n (x)= \frac{1}{d^*_{n}} f
((d_n x +x_n))$,  and  $d_n=d(x_n)=dist(x_n)$   and   $d^*_{n}=d(f(x_n))$  and \\
${\rm (b3)}$ for  every limit  $f_0$  of   $(f_n)$  in sense of
the  uniform convergence, $f_0$  is  a $C^1$ function   and
$J(f_0)$ has no zeros in  $G$. In this setting,  we say  that the
sequence $(f_n)$  is  associated sequence  to  the  sequence
$(x_n)$  and that  $f_0$  is  the   associated  limit.
\end{dfn}
%%%%%

%%%%%
\begin{thm}\label{thm.space2}
\rm{(a)}   Suppose  that   $\,h\,$ is    a euclidean harmonic K-qc mapping
from the unit ball $\mathbb{B}\subset \mathbb{R}^n$ onto  a
bounded convex  domain  $D=h(\mathbb{B})$, which
contains the  ball  $\,B(h(0);R_0)\,$.\newline
$\rm{(i.7)}$ If  $h$  has Jacobian non zero normal family-property  (or  $h\in  Q_H$),   then $h$ is  co-Lipschitz on  $\mathbb{B}$.\newline
$\rm{(i.8)}$ If  $\log J_h$  is $q$-superharmonic, then  $J_h\geq c$.
\end{thm}
\begin{proof}$\rm{(i.7)}$:
Suppose  that    there is sequence  $x_n$    such that
$|h'_r(x_n)|\rightarrow 0 $  and set  $d_n=d(x_n)$. Since by
Theorems  \ref{thm.space1}  and  \ref{aaaM1},
$\underline{J}(x_n)\succeq \frac{d_*}{d}  \succeq c$, there is  a
point $y_n\in \underline{B}(x_n)$  such that  $|h'_r(y_n)|\succeq
c$.  Apply a normal family argument on $h_n (x)= \frac{1}{d_{*n}} h
((d_n  x +x_n))$.\newline
$\rm{(i.8)}$  follows from Theorem  \ref{thm.space1},$\rm{(i.8)}$,    and Lemma  \ref{lem.AG0}.
\end{proof}
Using the Thom splitting lemma, we can prove

\begin{prop}\label{p-tom}
Suppose that $f$ is real-valued  function defined at a neighboorhood $ U(x_0)$   of a point  $x_0 \in \mathbb{R}^n$,   $f$  has partial derivatives up to the order  $3$  at $x_0$
and that   $f:  U(x_0)\rightarrow \mathbb{R}^n$  is  injective,
%K-qc,
where $U(x_0)$  is a neighborhood   of $x_0$ in   $\mathbb{R}^n$.  If   $\partial_k f(x_0)=0$,
then    $\partial^2_{ij}f(x_0)=0$,  $i,j=1,2,\cdots, n$, that is   ${\rm Hess(f)(x_0)=[0]}$.
%then  $K_O(f)  \geq 3^{n-1}$  on $U(x_0)$.
\end{prop}
Frequently  we use notation $X=(x,y,z)\in  \mathbb{R}^3$. If we
work in $\mathbb{R}^n$ it is convenient to switch the notation to
$x=(x^1,x^2, \cdots, x^n)\in \mathbb{R}^n$.
% =(x,y,z)
\begin{exa}\label{ex2rm} Let $a\neq 0$.
A radial mapping  $f_a$ in n-space  is given by:     $f(X) = f_a(X) =  |X|^{a-1}\,X$, where $X\in  \mathbb{R}^n$. Prove\\
$(i.2)$    $K_I(f)=|a|$, $K_O(f)=|a|^{n-1}$    if   $|a|\geq 1$;  in particular  $K(f_3)=K_O(f_3)=3^{n-1}$;\\
$K_I(f)=|a|^{1-n}$, $K_O(f)=|a|^{-1}$    if   $|a|\leq 1$.\\
In particular,  for $n=3$,  $X=(x,y,z)\in  \mathbb{R}^3$,  \\
$K_I(f)=|a|$, $K_O(f)=|a|^{2}$    if   $|a|\geq 1$;\\
 Now  we consider    3-space.\\
$(i.3)$   For $a=3$  set  $g=f_3 $; then    $\partial_k g(0)=0$,
$\partial^2_{ij}g(0)=0$  and   $g^1(X)= x^3  +   x y^2 + x
z^2$.\newline
%Suppose
$(i.4)$ \,For  $|a|\geq 1$,    $f_a$    is  co-Lipschitz on
$\mathbb{B}\setminus F$, where $F$  is a compact subset of
$\mathbb{B}$.

Set  $x'= f(x)$, $x=(x^1,x^2,x^3)$; then $ |x'|=|x|^{a}$  and   by
the cosine formula, we find $|x|^{2a} + |y|^{2a} - |x'-
y'|^2=|x|^{a-1} |y|^{a-1}( |x|^2 + |y|^2 - |x-y|^2 )$.
\begin{equation}
|x'- y'|^2=|x|^{a-1} |y|^{a-1}|x-y|^2  +R(x,y),
\end{equation}
where  $R(x,y)= |x|^{2a}  + |y|^{2a} - |x|^{a+1} |y|^{a-1}  -
|x|^{a-1}|y|^{a+1}$.

Without  loss  of  generality    we can suppose that   $|y|=
\lambda |x|$, $0\leq \lambda \leq 1$.  Then  $R(x,y)=  (1 -
\lambda^{a-1}) (1 - \lambda^{a+1})|x|^{2a}\geq 0\,$.  Therefore
$|x'- y'|^2\geq |x|^{a-1} |y|^{a-1}|x-y|^2\,$  and thus
\begin{equation}
|x'- y'|\geq \lambda^{(a-1)/2} |x|^{a-1} |x-y|\,.
\end{equation}
\end{exa}
%In joint work Mateljevi\'c and Bo\v zin have proved:
We need the following Proposition concernig the distortion
property of qr mappings.
\begin{prop}[\cite{vu}]\label{p-dist}
Let  $f:\mathbb{B}^n\rightarrow \mathbb{B}^n$  be  K-qr, $f(0)=0$
and $\alpha=K_I(f)^{1/(1-n)}$.  Then  \newline $(i.5)$  $
|f(x)|\leq \varphi_{K,n}(|x|)\leq \lambda_n^{1-\alpha}
|x|^\alpha$. \newline $(i.6)$  If  $g:\mathbb{B}^n\rightarrow
\mathbb{B}^n$  is   K-qc,  $g(0)=0$  and $1/\alpha
=K_I(g^{-1})^{1/(n-1)}$, then   \\
$m |x|^{1/\alpha}\leq |g(x)|$.
\end{prop}
Suppose that $g$   is analytic (more generally $C^{(3)}$ at $0$),
$g(0)=0$, $\partial_k g(0)=0$ and   $\partial^2_{ij}g(0)=0$.  Then
$|\partial_k g(x)|\leq M |x|^2$  and  therefore   $|g(x)- g(y)|
\leq M |x|^2 |x-y|$ if $|y| \leq |x|$.  In particular,  $|g(x)|
\leq M |x|^3$.\newline Note that,   if  in addition, $f$  is
$C^{(4)}$ at $0$  and  $\partial^3_{ijk}g(0)=0$, then  $|g(x)-
g(y)| \leq M |x|^3 |x-y|$ if $|y| \leq |x|$ and $f_a +g$  is K-qc
for $a <4$.  In particular,  $|g(x)| \leq M |x|^4$.
\begin{prop}\label{p-tom2}
Suppose that $f$  has partial derivatives up to the order  $3$  at
a point  $x_0 \in \mathbb{R}^n$ and that   $f:  U(x_0)\rightarrow
\mathbb{R}^n$  is  K-qc, where $U(x_0)$  is a neighborhood   of
$x_0$ in   $\mathbb{R}^n$.  If   $\partial_k f(x_0)=0$ and
$\partial^2_{ij}f(x_0)=0$, then  $K_O(f)  \geq 3^{n-1}$  on
$U(x_0)$.
\end{prop}
Example \ref{ex2rm} shows that the result is optimal.
\begin{proof} By the Taylor formula,  there is  $M$   such  $|f(x)|
\leq M |x|^3$.\newline If  $1/\alpha =K_I(g^{-1})^{1/(n-1)}$, then
by Proposition \ref{p-dist},\,   $m |x|^{1/\alpha}\leq |g(x)|\leq
M |x|^3$.   Hence   $K_O^{1/(n-1)} \geq 3$, and therefore  $K_O
\geq 3^{n-1}$, where $K_O(g)=K_I(g^{-1})$.
\end{proof}
\begin{prop}\label{p-tom3}
Suppose that $f$  has continuous partial derivatives up to the
order  $3$  at the origin $0$  and that   $f:  U(0)\rightarrow
\mathbb{R}^n$  is  K-qc, where $U(0)$  is a neighborhood   of $0$
in   $\mathbb{R}^n$. If  $K_O(f) < 3^{n-1}$, then  $J(f,0)\neq 0$.\\
In particular,  if  $g$   is analytic \rm{(}more generally
$C^{(3)}(U(0))$   or $g$  only has partial derivatives up to the
order  $3$\rm{)}, and  if  $g$   is   K-qc   with     $K_O(g) <
3^{n-1}$, then  $J(g,0)\neq 0$.
\end{prop}
\begin{proof}
Contrary suppose that $J(f,0)=0$. Since $f$  is   K-qc,
$\partial_k f(x_0)=0$,  hence by  Proposition  \ref{p-tom}  we
find $\partial^2_{ij}f(x_0)=0$.  Now,  by  Proposition
\ref{p-tom2}, $K_O(f)  \geq 3^{n-1}$  and this yields a
contradiction.
\end{proof}
% continuous the origin  $f(0)=0$,  and
%$g(0)=0$
%In joint work with Bo\v zin, we prove
\begin{thm}
If in  addition to hypothesis  \rm{(a)} of  Theorem
\ref{thm.space2}  we suppose that    $\,h\,$ is qc  with  $K_O(h)
< 3^{n-1}$, then $h$ is  co-Lipschitz on  $\mathbb{B}$.
\end{thm}
\begin{proof} The proof follows from Theorem \ref{thm.space2} and  Proposition \ref{p-tom3}.  We leave the details to the interested reader.
\end{proof}
%%%%%% AA
%\cite{km.pac}

\section{The Lower Bound of the Jacobian in $\mathbb{R}^3$}\label{grad3}
%  gradient mappings of
It seems  a natural project  to   generalize and develop the arguments used in planar theory of
harmonic mappings to  harmonic functions in
domains of $\mathbb{R}^n$,  $\geq 3$.
% $h$
In Section \ref{sec.pl} we used  the fact that  every   complex
harmonic function $h$ on simple connected planar domain, can be
written   in the form    $h=f+\overline{g}$, where $f$ and $g$ are
holomorphic that  $|f'| $  satisfies minimum principle.
% gradient
There is no appropriate analogy of this result in space.   The next  example shows that the minimum principle  does not hold  for modulus of  vector valued harmonic mapping.
\begin{exa}
Define  $h(M)=  (x,y, x^2 +y^2-1 - 2z^2)$,  where  $M=(x,y,z)$,
$f_0(x,y)=( x,y,x^2+y^2 -1)$ and  $G=\{(x,y,z): z < x^2 + y^2
-1\}$. Then  $J(h)= -4z$, the restriction of $h$ on xy-plane is
$f_0$,   $\Gamma_{f_0} =\partial G$  and  $h(\mathbb{R}^3)=G$.
%If, then  $J_f=-2y$.Let $d(z)= |z|$, $g=f_{-1}$ and.
Let $d(M)= |OM|$.
Then $d$  attains minimum on
$\Gamma_{f_0} $ at some point $M_1$  and there is a  point
$M_0= (x_0,y_0,0)$ such that $f_0(M_0)=M_1=h(M_0)$, $h$ maps
$\mathbb{R}^3$ onto $G$ and therefore $|h|$ attains minimum at  $M_0$.
\end{exa}

If $h$ is a  harmonic
mapping from a domain in $\mathbb{R}^{n}$   to $\mathbb{R}^n$,
then   $|h'_{x_k}| $  is  subharmonic,   but it does not satisfy  minimum principle  in general (adapt  the above example to the dimension $n\geq 3$).

In fact,  Lewy's theorem is false
in dimensions higer than two (see \cite{dur}  p. 25-27 for Wood's
counterexample).

Consider the polynomial map from   $\mathbb{R}^3$  to
$\mathbb{R}^3$  defined by  $h(x,y,z)=(u,v,w)$, where

$$u = x^3 -3xz^2 +  yz,\quad v=y-3xz, \quad  w=z\,. $$
a calculation shows that $h$  has the Jacobian
$$ J_h(x,y,z)= 3x^2,$$
%\,
which vanishes on the plane $x=0$.
Jacobian of $C^1$  orientation preserving mapping $f$ is
nonnegative.
Iwaniec \cite{Iw14} suggest a project (for students): Generalize
and develop the arguments used in planar theory of harmonic
mappings to gradient mappings of harmonic functions in domains of
$\mathbb{R}^3$.  Recall that, in planar theory of   harmonic
mappings we used

$\rm{(I0)}$ version of Hall   and

$\rm{(II.0)}$ If analytic function does not vanish then its
modulus satisfies minimum principle.

For harmonic   gradient mapping  in $3$-space    Proposition
\ref{pro3.1} and Theorem \ref{t.gl} are analogy of $\rm{(I0)}$ and
$\rm{(II.0)}$  respectively.

%PROJECT, For Students

For $n = 3$, Lewy proved that the Hessian of a harmonic function
(the determinant of its matrix of second derivatives) cannot
vanish at an interior point of its domain without changing sign,
unless it vanishes identically. More precisely, if the Hessian
vanishes at some interior point $x_0$ without vanishing
identically, then in each neighborhood of $x_0$  it must take both
positive and negative values. But the Jacobian of a harmonic
mapping $f = {\rm grad}\,u$  is the Hessian of $u$.

As a consequence, the Jacobian of a locally univalent harmonic
gradient mapping from $\mathbb{R}^{3}$   to $\mathbb{R}^{3}$
cannot vanish at any interior point of its domain.  Gleason and
Wolff \cite{gl}   generalized this result to $\mathbb{R}^{n}$.

Throughout this text the subscripts with variables $x, y$  and $z$
designate partial derivatives.

A vector field  $F = (f^1, f^2, f^3)$ is said to satisfy the
Cauchy-Riemann equations (CR-equations, for short) if its
coordinates, (conjugate harmonic functions) satisfy:

$f^1_y=f^2_x,\, f^1_z=f^3_x,\, f^2_z=f^3_y,\,f^1_x+
f^2_y+f^3_z=0$, locally    $F= {\rm grad} \phi  $  and  $\triangle
\phi =0 $.

Equivalently, the Jacobian matrix of $F$  is symmetric and has
trace $0$.

\begin{thm}[Lewy-Gleason-Wolff,\cite{gl}] \label{t.gl}   Logarithm of modulus of the Hessian of a harmonic function
in a domain $\Omega \subset  \mathbb{R}^3$ is superharmonic
outside its zeros. Precisely, $\Delta \ln |H|\leq 0 $,  wherever
$H\neq 0$.
\end{thm}
Remark.  This inequality fails in dimensions greater than $3$.
Obviously, it holds (as equality) for planar harmonic functions.

\begin{prop} \label{c.hess}  Suppose Hessian determinant $H$ of a harmonic
function in a domain   $\Omega \subset  \mathbb{R}^3$   is
positive.  Then for every compact  $F \subset \Omega$
 we have  $\inf_ F H \geq inf _{\partial F} H$.

In particular,  if   $f$ is  injective  harmonic gradient mapping,
we have $\inf_ F J(f) \geq inf _{\partial F} J(f)$.
\end{prop}
%%%%%%
%%%%%%
%\subsection{Hall lemma  and  co-Lipschitz property of qc gradient mappings}

Using a normal family argument  one can prove  (see subsection
\ref{ss.hall} for details):
\begin{prop}\label{pro3.1}\cite{mm.unpub2}
Suppose that $G$ and $G^{\prime }$ are domains in $\mathbb{R}^{3}$: If $%
f:G\rightarrow G^{\prime }$ is  injective   harmonic K-qc gradient
mapping, then $f$ has   weak $H$-property.
\end{prop}
A more general result will appear in a forthcoming paper.
%If $f$ is  univalent harmonic K-qc gradient mapping from
% $\mathbb{B}$ onto convex domain  $D\subset \mathbb{R}^{3}$, then   $f$  has   $H$-property.
%%%%%

%%%%%
\begin{prop}
Suppose that

$(a1)$:   $f$ is  injective  harmonic gradient mapping from
$\mathbb{B}$ onto $D\subset \mathbb{R}^{3}$  and

$(b1)$:  partial derivatives of  $f$ have  continuous extension to
$\overline{\mathbb{B}}$

$(i.1)$  If $J_h \geq j_0>0 $ on $\mathbb{S}^2$, then   $J_h \geq j_0$  on
$\overline{\mathbb{B}}$.

$(i.2)$    If  in addition $f$ is K-qc,    then  $f:  \mathbb{B} \rightarrow
D$ is bi-Lipschitz.
\end{prop}

\begin{thm} \label{3clip}
Suppose that

$(a1)$  $f$ is continuous  on  $\overline{\mathbb{B}}$,   and
univalent harmonic K-qc gradient mapping from $\mathbb{B}$ onto
convex domain  $D\subset \mathbb{R}^{3}$  and

$(b1)$   partial derivatives of  $f$ have  continuous extension to
$\overline{\mathbb{B}}$.

$(ii.1)$  Then  $f:  \mathbb{B} \rightarrow D$ is bi-Lipschitz   and in
particular  $f^{-1}: D \rightarrow \mathbb{B}$  is L-Lipschitz.

$(ii.2)$ If $f$ is  injective  harmonic K-qc gradient mapping from
$\mathbb{B}$ onto convex bounded domain  $D\subset
\mathbb{R}^{3}$, then $f$ is  co-Lipschitz.
\end{thm}

\begin{proof}
$(i.1)$, $(i.2)$ and $(ii.1)$   are    corollary of Proposition \ref{c.hess}.

$(ii.2)$ is a corollary of Proposition \ref{pro3.1}  and Theorem \ref{thm.space1} $(i.3)$.
\end{proof}
% Proposition\ref{pro3.1} and
Note that the  above outline  of proof of  $(ii.2)$ is not based
on Theorem \ref{t.gl}  (see subsection \ref{ss.hall} for more
details).
Astala-Manojlovi\' c first  made publicly available proof of
$(ii.2)$ in Math.Arxiv, \cite{ast.ma}.
%; see  also Mateljevi\'c   manuscript  in Math.Arxiv,  \cite{rckm0}).

In particular,  $(ii.2)$    yields:

\begin{prop} (b) Suppose that $f$ is  univalent harmonic gradient
mapping from $\mathbb{B}^3$ onto itself. Then $f$ is
co-Lipschitz.
\end{prop}
%%%%
%O. Martio  \cite{OM1} proved this in planar case.
In communication between   V.  Zorich  and the author,  the
question  was asked to find examples of functions that satisfy the
condition (c). For example,  if $u=x^2 +y^2 - 2 z^2$, then   $f=
\nabla u=(2x,2y,-4z)$ is injective  harmonic gradient mapping from
$\mathbb{B}^3$ onto the ellipsoid.

If $u$   is real-valued function   such that   $f= \nabla
u=(x,y,z)$, then   $u=  x^2/2 +y^2/2 +z^2/2 +c$.

In particular, $Id$   is not  harmonic gradient mapping.

In complex plane,  if  $u$   is real-valued harmonic function, then     $u_z=\frac{1}{2}(u'_x -u'_y)$  is analytic function  and therefore
$ \nabla u=\overline{F}$, where  $F= 2 u_z$  is analytic function.

\subsection{Hall lemma  and  co-Lipschitz property of qc gradient
harmonic mappings}\label{ss.hall}
%%%%%

Here,  we  outline a  proof  of  Theorem \ref{3clip},  the part
$(ii.2)$, stated here as:

(A0)   A euclidean gradient harmonic mapping from the unit ball
$\mathbb{B}\subset \mathbb{R}^3$ onto  a bounded convex  domain is
co-Lipschitz.

% $(ii.1)$
%Recall Astala \& Manojlovi\' c first  made publicly available
%proof of the statement (A0)  in Math.Arxiv, \cite{ast.ma}.

Note that our proof of  (A0)  is based on the Hall lemma  and a
normal family argument  and our approach  is different from that
in \cite{ast.ma}; see also \cite{BoMa2}.

We first  prove Hall lemma for  harmonic injective mappings in
n-dimensional space.

If  $G$ and $G'$ are domains in $ \mathbb{R}^n$,  by $QCH(G,G')$
(respectively  $QCH_K(G,G')$ ) we denote the set of Euclidean
harmonic quasiconformal mappings (respectively K-qc)  of $G$ onto
$G'$.
\begin{lem}\label{lhall1}
Let $F_K$  be  a family of   harmonic  K-qc mapping $f:\mathbb{B}
\rightarrow \mathbb{R}^n$ such  that for $f\in  F_K$, $J(f)$  has
no zeros, $F_K$  is closed  with respect to uniform convergence,
$f(\mathbb{B})\supset \mathbb{B}$ and  $f(0)=0$. Then  there is a
constant $c>0$  such that if $f\in  F_K$   is harmonic  K-qc
mapping $f(\mathbb{B})\supset \mathbb{B}$, $f(0)=0$, then
$\Lambda_f(0) \geq c$.
\end{lem}
\begin{proof} Suppose, on the contrary,  that   there is
a  sequence $f_n\in  F_K$   such that $\Lambda_{f_n}(0)\rightarrow
0$.  The sequence $(f_n)$    forms a normal family and there is a
subsequence of $f_n$ which  converges uniformly  to a limit
$f_0\in  F_K$; this is a contradiction.
\end{proof}

Our further considerations  are  related  to  weak $H$-property
(see Definition  \ref{dfn.wH}).

 By Hall  lemma  planar euclidean harmonic mappings  have   $H$-property.
 %%%%%
 Let  $QCH^0_K= QCH^0_K(\mathbb{B},D)$ be family of  gradient harmonic mappings which maps $\mathbb{B}$  onto $D$.
%%%%
\begin{dfn}
We say that   $\underline{F}\subset QCH_K(G,G')$   has
$J$-property   if $F$  is closed  with respect to uniform
convergence, and for $f\in  \underline{F}$, $J(f)$  has no zeros
in $G$.
\end{dfn}
%%%%%

%%%%
Now we sketch  a proof of the statement (A0)   in few steps
(A1-A5):

(A1)   If  $\underline{F}\subset QCH_K(G,G')$   has $J$-property,
then   $f\in\underline{F}$  has   weak $H$-property.

Outline of proof.  By Lemma  \ref{lhall1}, there is  a constant
$c>0$  such that $d(x) \Lambda_f(x)\geq c d(fx)$,  $x \in G$.

(A2) Suppose  that   $\,h\,$ is    a euclidean harmonic mapping
from the unit ball $\mathbb{B}\subset \mathbb{R}^n$ onto  a
bounded convex  domain  $D=h(\mathbb{B})$. Then   there is a
constant $c>0$ such that   $d^*(f(x)) \geq c d(x)$,  $x \in
\mathbb{B}$.
% \Lambda_f(x)

(A3)   If  $f$ is injective  harmonic gradient mapping, which maps
$\mathbb{B}\subset \mathbb{R}^n$ onto  a bounded convex domain
$D=h(\mathbb{B})$, then  any associated limit of $f$  is harmonic
gradient mapping.\\
In dimension $n=3$,  then it  has Jacobian non zero normal
family-property.

(A4)   If  $\underline{F}\subset QCH_K(B,D)$   has $H$-property,
then  every  $f\in\underline{F}$  is   co-Lipschitz.

By  (A1)   there is a constant $c>0$   such that  $d(x)
\Lambda_f(x)\geq c d^*(f(x))$, $x \in \mathbb{B}$, and by  (A2), a
constant $c_1>0$  such that   $d^*(f(x)) \geq c_1  d(x)$. Hence
$\lambda_f(x)\geq c_2$, where  $c_2=c  c_1/K$.

(A5) In 3-dimensional  space,    $QCH^0_K= QCH^0_K(\mathbb{B},D)$
has $H$-property.

 From (A1-A5), it follows   (A0).

%(A0) Let $D$  be a convex  domain.  In particular, every  $f\in QCH^0_K= QCH^0_K(\mathbb{B}_3,D)$  is  co-Lipschitz.

%%%%%

\section{Appendix 1}\label{app}
%%%%%
\medskip
In this review  section we follow \cite{revroum01,napoc1}. First
we recall some results from Section \ref{sec.pl} (Theorem
\ref{t.in.belowA}) and prove  $\rm{(I0)}$    version of Hall
lemma.
%with constant
%As a corollary of Proposition 1 we obtain

\smallskip
\begin{thm}\label{t.in.belowB}
{\it Let  $h$ be an euclidean harmonic orientation preserving
univalent mapping of the unit disc onto convex domain $\Omega$. If
$\Omega$ contains a disc $B(a;R)$ and $h(0)=a$ then \[|\partial
h(z)|\geq\frac{R}{4},\,\,z\in \mathbb{D}.\]}
\end{thm}

\medskip
As a corollary of the previous Theorem  we obtain

\begin{thm}
{\it Let  $h$ be an euclidean harmonic orientation preserving
$K$-qc mapping of the unit disc onto convex domain $\Omega$. If
$\Omega$ contains a disc $B(a;R)$ and $h(0)=a$ then \[|\partial
h(z)|\geq\frac{R}{4},\,\,z\in  \mathbb{D},\]}

  $$ l_h (z)\geq \frac{1-k}{4}R .$$
\end{thm}
Let  $c=\frac{1-k}{4}R$.   Since   \[|\overline{D} h(z)| \leq
k\,|Dh(z)|,\]   it follows   that  $ l_h (z)\geq c
=\frac{1-k}{4}R$ and therefore     $ |h(z_2)- h(z_1)|\geq c\,|z_2
-z_1|$.
%\geq\frac{R}{4}

As a corollary of Theorem \ref{t.in.belowB} we obtain
\smallskip
\begin{prop}\label{prop4.1}
{\it Let  $h$ be an euclidean harmonic orientation preserving
univalent mapping of the unit disc into $\Bbb C$ such that $f(\mathbb{D})$
contains a disc $B_R=B(a;R)$ and $h(0)=a$. Then \begin{equation}
|\partial h(0)|\geq\frac{R}{4}. \end{equation}}
\end{prop}

\smallskip
\noindent \it Proof: \rm Let $V=V_R=h^{-1}(B_R)$ and $\varphi$ be
a conformal mapping of the unit disc $U$ onto $V$ such that
$\varphi(0)=0$ and let $h_R=h\circ \varphi$. By Schwarz lemma
\begin{equation} |\varphi'(0)|\leq 1. \end{equation}
Since $\partial h_R(0)=\partial h(0) \varphi'(0)$, by Proposition
\ref{prop4.1}    we get $|\partial h_R(0)|=|\partial
h(0)||\varphi'(0)|\geq \frac{R}{4}.$ Hence, using (0.2) we get
(0.1).

\medskip
Also as an immediate  corollary of Theorem \ref{t.in.belowB}  we
obtain

\smallskip
\noindent
\begin{thm} [\protect\cite{k,Dk}]
{\it Let $\,h\,$ be an euclidean harmonic diffeomorphism of the
unit disc onto convex domain $\,\Omega$. If $\,\Omega\,$ contains
a disc $\,B(a;R)\,$ and $h(0)=a$ then
$$
D(h)(z)\geq\frac{1}{16}R^{2}, z \in \mathbb{D},
$$
where $\,D(h)(z)=|\partial h(z)|^{2}+|\overline\partial
h(z)|^{2}$.}
\end{thm}

\medskip
\noindent The following example shows that previous results
%Theorem 1, Proposition1, Theorem 2, Proposition 2 and Theorem 3
are not true if we omit
the condition $h(0)=a$.

\medskip
\noindent {\bf Example.} The mapping
$$
\varphi_{b}(z)=\frac{z-b}{1-{\bar b}z},\,\,|b|<1,
$$
is a conformal automorphism of the unit disc onto itself and
$$
|\varphi_{b}'(z)|=\frac{1-|b|^{2}}{|1-{\bar b}z|^2},\,\,z\in  \mathbb{D}.
$$
In particular $\varphi_{b}'(0)=1-|b|^{2}$.

\bigskip
Heinz proved (see  \cite{H}) that if $\,h\,$ is a harmonic
diffeomorphism of the unit disc onto itself such that $\,h(0)=0$,
then
\[
D(h)(z)\geq \frac{1}{\pi ^{2}},\,\,z\in  \mathbb{D}.
\]

\medskip
Using Proposition \ref{prop4.1} we can prove Heinz theorem:

\smallskip
\noindent
\begin{thm}[\bf Heinz\rm]
{\it There exists no euclidean harmonic diffeomorphism from the
unit disc $ \mathbb{D}$ onto $ \mathbb{C}$.}
\end{thm}

\smallskip
Note that this result was a key step in his proof of the Bernstein
theorem for minimal surfaces in $\,\Bbb R^3\,$.

\medskip
Schoen obtained a nonlinear generalization of Proposition \ref{prop4.1}  by
replacing the target by complete surface of nonnegative curvature
(see Proposition 2.4 \cite{Sc}) and using this result he proved

\smallskip \noindent
\begin{thm}[\bf Schoen\rm]
{\it There exists no harmonic diffeomorphism from the unit disc
onto a complete surface $(S, \rho)$  of nonnegative curvature
$K_\rho \geq 0$.}
\end{thm}
Suppose $f$ is  a harmonic diffeomorphism from $B_r$ to $ ( S,
\rho ) $ and  $ dist ( f(0), \partial (f(B_r)) \geq R$. Then it
suffices to show that
$$  |df|^2 ( 0 ) \geq  C \, \frac{R^2}{r^2}\,\,,$$
where  $C$ is a universal constant. By hypothesis, we have
$|\partial f| > | \overline{\partial} f| \geq  0$  and
\begin{equation}\label{bochn.1}
\Delta  ln  |\partial f| = -\, K_\rho \, J_f  \leq  0\, .
\end{equation}
If we define  a Riemannian  metric  $\lambda$ on $B_r$  by
$\lambda = |\partial f|^2 \,|dz|^2,$  then  (\ref{bochn.1})
implies  $ K_\lambda\geq 0$.  Therefore $dist (0,
\partial (B_r )  \geq \frac{1}{2}  dist (f(0), \partial ( f (B_r) )   \geq \frac{1}{2}  R$

\begin{lem}\label{est.curv}
If $\sigma$ is a metric density  of nonnegative curvature
$K_\sigma \geq 0$ on $B_r$  and  $ d= dist_\sigma\, (0, T_r)$,
then   $ \sigma (0)\geq  C \, \frac{d^2}{r^2}\,\,,$ where  $C$ is
a universal constant.
\end{lem}
 A proof can be given by means the estimate of harmonic  function in terms of
curvature  (Cheng-Yau,  CPAM 28, 333-354 (1975)). We apply this
lemma to metric density $ \lambda = |\partial f|^2$. By the above
estimate,
$$  |\partial f|^2 ( 0 ) \geq  C \, \frac{R^2}{r^2} .$$
This proves the theorem.

{\it  Question}\,2. Can we prove Lemma \ref{est.curv} elementary?  Note that
$ ln \sigma$ is superharmonic function. Therefore $ ln
\frac{1}{\sigma}$ and  $\frac{1}{\sigma}$ are subharmonic
functions.

\subsection{Distortion  of conformal mappings} The following form
of Koebe's One-Quarter Theorem applies in fact to all conformal
mappings.
\begin{thm} Suppose that $f$ is bijective conformal in $D$  and $f(D)= D'$,
$z_0\in D$. Then
$$\frac{1}{4}|f'(z_0)| {\rm dist}(z_0,\partial D)\leq {\rm
dist}(f(z_0),\partial D') \leq 4 |f'(z_0)| {\rm dist}(z_0,\partial
D).$$
\end{thm}
If we set $d=d(z_0)={\rm dist}(z_0,\partial D)$ and
$d'=d'(f(z_0)={\rm dist}(f(z_0),\partial D')$  then\\
${\rm
(iv.1)}$: $d' \asymp d$.

 \bpf Let $d=d(z_0)=d_D(z_0)={\rm
dist}(z_0,\partial D)$,   $d'=d(f (z_0))=d_{D'}(f (z_0))={\rm
dist}(f(z_0),\partial D')$;
$$g(z)=\frac{f(z)-f(z_0) }{d\cdot f'(z_0)}\ {\rm and}\ f_0(z)= g(z_0 + zd).$$
Set  $D_0=g(\mathbb{D}(z_0;d))$ and $d_0={\rm
dist}(g(z_0),\partial D_0)$. Note that  $d_1={\rm
dist}(g(z_0),\partial D)=\frac{d'}{d|f'(z_0)|}$  and  $d_1 \geq
d_0$. Since  $f'_0(0)=1$,  it follows from Koebe's One-Quarter
Theorem,  applied to  $f_0$ that $d_0\geq 1/4$. Hence,  since
$\frac{d'}{d|f'(z_0)|}=d_1 \geq d_0\geq  1/4$, we get the left
inequality.

Koebe's  Theorem  applied to   $f^{-1}$ at  $w_0=f(z_0)$  gives
$\frac{1}{4}|(f^{-1})'(w_0)|d'\leq d$ and  the right inequality
follows. \epf

If we define  $D_f^*(z_0)=\frac{d |f'(z_0)|}{d'}$, we can
reformulate  the above theorem as   $1/4 \leq D_f^*(z_0)\leq 4$.
 The interested reader can check that
$$D_{f^{-1}}^*(w_0) = D_f^*(z_0).$$

The following two basic theorems are important  for our research.

%thm  Kellogg  \label{oneone}
\begin{thm}[Kellogg, see for example \cite{G,gol}]\label{oneone} {\it If a domain
$D=\rm{Int}(\Gamma)$ is $C^{1,\alpha}$, $0 < \alpha <1$,  and
$\omega$ is a conformal mapping of $\mathbb{D}$ onto $D$, then
$\omega'$ and $\ln \omega'$ are in $\rm{Lip}_\alpha.$ In
particular, $|\omega'|$ is bounded from above and below on
$\mathbb D$}.
\end{thm}

\begin{thm}[Kellogg and Warschawski, see \cite{w2}, Theorem 3.6] \label{onetwo} {\it If a domain
$D=\rm{Int}(\Gamma)$ is $ C^{2,\alpha}$ and $\omega$ is a
conformal mapping of $\mathbb{D}$ onto $D$, then $|\omega''|$ has
a continuous extension to the boundary. In particular it is
bounded from above on $\mathbb D$}.
\end{thm}
\subsection{The uniformization theorem}
The uniformization theorem says that every simply connected Riemann surface is conformally equivalent to one of the three domains: the open unit disk, the complex plane, or the Riemann sphere. In particular it admits a Riemannian metric of constant curvature. This classifies Riemannian surfaces as elliptic (positively curved – rather, admitting a constant positively curved metric), parabolic (flat), and hyperbolic (negatively curved) according to their universal cover.

The uniformization theorem is a generalization of the Riemann mapping theorem from proper simply connected open subsets of the plane to arbitrary simply connected Riemann surfaces.
The uniformization theorem implies a similar result for arbitrary connected second countable surfaces: they can be given Riemannian metrics of constant curvature.
Every Riemann surface is the quotient of a free, proper and holomorphic action of a discrete group on its universal covering and this universal covering is holomorphically isomorphic (one also says: "conformally equivalent") to one of the following:
the Riemann sphere,  the complex plane
or    the unit disk in the complex plane.  Koebe proved the general uniformization theorem that if a Riemann surface is homeomorphic to an open subset of the complex sphere (or equivalently if every Jordan curve separates it), then it is conformally equivalent to an open subset of the complex sphere.
In $3$  dimensions, there are $8$  geometries, called the eight Thurston geometries. Not every $3$-manifold admits a geometry, but Thurston's geometrization conjecture proved by Grigori Perelman states that every $3$-manifold can be cut into pieces that are geometrizable.
The simultaneous uniformization theorem of Lipman Bers shows that it is possible to simultaneously uniformize two compact Riemann surfaces of the same genus $>1$  with the same quasi-Fuchsian group.  The measurable Riemann mapping theorem shows more generally that the map to an open subset of the complex sphere in the uniformization theorem can be chosen to be a quasiconformal map with any given bounded measurable Beltrami coefficient.
%%%
\section {Appendix 2,  Harmonic Maps Between surfaces}\label{app2}

In \cite{jost2} it is given  self -contained account of the
results on harmonic maps between surfaces.
%Nevertheless, I believe that
This treatment contains several simplifications and unifications
compared to the presentations available in the existing
literature.
%%%%%
%%%%%
Upper and lower bounds for the sectional curvature $K$  of a
manifold are  often denoted by  $k^2$  and   $-\omega^2$,  i.e.
$-\omega^2 \leq K \leq k^2$.  This notation avoids square roots.
It differs, however, from the terminology in some of the papers
frequently referred to in the  book \cite{jost2}.
%Theorem 7.2:  present

Here we give short review  of results related to our consideration
in  section \ref{sec.pl}.  First  the lower bound for Jacobinan
are considerd.
\begin{thm}\label{7.2}
Suppose  $u: \mathbb{D} \rightarrow \Sigma$ is harmonic, and $u(D)
\subset B(p,M)$, where  $B(p,M)$  again is a disc with radius $M <
\pi /2k$. Suppose that $\partial u(\mathbb{D}) = u(\mathbb{T})$
and that $g: = u|T : T
\rightarrow  \partial u(D)$  is a   $C^2$- diffeomorphism with\\
%(7.3.1)
${\rm (d1)}$  $0 < b \leq |g'(t)|$ for all $t \in [0,2\pi]$.\\
Assume furthermore that $g(\mathbb{T})$ is strictly convex w.r.t,
$u(\mathbb{D})$, and that we have the following estimates for the
geodesic curvature of
$g(\mathbb{T})$\\
${\rm (d2)}$   $0 < a_1  \leq \kappa_g (\underline{g}(t)) \leq
a_2$   for all $t \in [0,2\pi]$.

Then  ${\rm (iv.1)}$:

$J \geq \delta_1^{-1}$, where    $\delta_1 = \delta_1(\omega,k, M,
\tau,a_1,a_2,b,|g|_{1,\alpha})$  ($\tau$   is given in Thm. 6.2
\cite{jost2}).
\end{thm}
%we did not even have to
Note that  it is not assumed in the theorem  that $u$  is
univalent; we needed only that $u$ maps $\mathbb{D}$  onto the
convex side of $u(\mathbb{T})$. If $u$ is an injective harmonic
map, then $J \geq \delta_2^{-1}$ on $\mathbb{D}$.
% 7.3.1)=d1 7.3.1) (7.3.2)=d2
\begin{prop}\label{p7.2}
Assume $u: \mathbb{D} \rightarrow \Sigma$  is an injective
harmonic map, where    $u(D) \subset B(p,M)$, and  $B(p,M)$ is a
disc with radius $M < \pi /2k$. Suppose that   $g: =
u|\mathbb{T}\in C^{1,\alpha}$, and that  ${\rm (d1)}$  and ${\rm
(d2)}$  hold. Then   ${\rm (iii.2)}$:  for all $z\in\mathbb{D}$,
$J \geq \delta_2^{-1}$, where $\delta_2 = \delta_2(\omega,k, M,
\tau,a_1,a_2,b,|g|_{1,\alpha})$.
\end{prop}

We define   $d_*(q): = -dist(q,\partial u(D))$ for   $q \in u(D)$.
Since $\triangle (d_*\circ u) \geq a_1 b^2$,  \,  $d_*\circ u$ is
subharmonic function. This will enable us to get a lower bound for
the radial derivative of \,$d_*\circ u$ \,  at boundary points
with the argument of the boundary lemma of E. Hopf.

Taking Cor. 6.2\cite{jost2} into account, we can therefore find a
neighborhood $V_0$  of $\mathbb{T}$  in  $\mathbb{D}$  with the
property that $d_*$  is a  $C^2$ function with strictly convex
level curves on $u(V_0)$. Suppose $z_0\in \mathbb{T}$; we can
choose some disc $B_1= B(z_1,r_1) \subset D$,   $z_0\in
S(z_1,r_1)$, in such a way that  $\triangle (d_*\circ u) \geq a_1
b^2/2$  for  $z\in B_1$.

Defining the auxiliary function $v$  via    $v(z)=
\frac{r_1^2}{8}a_1 b^2 (1- \frac{|z-z_1|^2}{r_1^2})$, we find
$\triangle v= - a_1 b^2/2$ and therefore  $\triangle (d_*\circ u +
v)\geq 0$. The maximum principle now controls the derivative of \,
$d_*\circ u + v$\,  at  $z_0$ in  the direction of the outer
normal.
%, namely${\rm (b1)}$ Theorem 8.1:

Now,using the above estimates,  the existence of harmonic
diffeomorphisms which solve a Dirichlet problem is considered.
\begin{thm}\label{8.1}
Suppose ${\rm (e1)}$: $\Omega$ is a compact domain with Lipschitz
boundary $\partial \Omega$   on some surface, and that $\Sigma$ is
another surface. We assume ${\rm (e2)}$: that $f:
\overline{\Omega} \rightarrow \Sigma$ maps $\overline{\Omega}$
homeomorphically onto its image, that $f(\partial \Omega)$  is
contained in some disc $B(p,M)$ with radius $M < \pi /2k$  (where
$k^2 \geq 0$ is an upper curvature bound on $B(p,M)$) and that the
curves $f(\partial \Omega)$  are of Lipschitz class and convex
w.r.t. $f(\Omega)$.
% ${\rm (c1)}$:
Then ${\rm (v.1)}$: there exists a harmonic mapping $u: \Omega
\rightarrow B(p,M)$  with the boundary values prescribed by $f$
which is a homeomorphism between $\overline{\Omega}$  and its
image, and a diffeomorphism in the interior.\\
${\rm (v.2)}$   Moreover, if $f|\partial \Omega$ is even a $C^2$
-diffeomorphism between $C^2$ -curves, then $u$ is a
diffeomorphism up to the boundary.
\end{thm}

First of all, $\partial \Omega$  is connected. Otherwise,
$f(\partial \Omega)$ would consist of at least two curves, both of
them convex w.r.t. $f(\Omega)$.
%topologically.  and
Since
$\Omega$  is homeomorphic to $f(\Omega)$,  we conclude that
$\Omega$    is a disc, topologically.  Since  there is a conformal map   $\phi : D\rightarrow g(\mathbb{D})$,
%Therefore,
one have to
prove the theorem only for the case where $\Omega$   is the plane
unit disc $\mathbb{D}$.

We first assume that  $f:\mathbb{T}\rightarrow f(\mathbb{T})$  is  $C^2$
-diffeomorphism between curves of class   $C^{2,\alpha}$, that
$f(\mathbb{T})$  is not only convex, but strictly convex, and that
we have the following
quantitative bounds \\
${\rm (e2)}$  \,     $|\underline{f}'(t)|\geq b_2^{-1}$  and  $
|\underline{f}''(t)| \leq b_1$.

Now let  $\Gamma$   be the parametrization of the boundary curve
of $g(D)$   by arclength. If   $l$ is length of  $\partial g(D)$,
then  $\Gamma$  maps $[0,l]$   on  $\partial g(D)$. We set
$w(z,\lambda)=  \lambda \Gamma^{-1}(\phi(z) + (1- \lambda)
\Gamma^{-1}(g(z))$,   $\omega(z,\lambda)= \Gamma (w(t,\lambda)) $,
$z\in T$, $\lambda \in [0,1]$, and $\underline{\omega}(t,\lambda)
\omega(e^{it},\lambda)$. Using  ${\rm (e2)}$, one can check that

$\underline{\omega}(t,\lambda)$, $\frac{\partial
\underline{\omega}(t,\lambda)}{\partial t}$ and $\frac{\partial^2
\underline{\omega}(t,\lambda)}{\partial^2 t}$  are continuous
functions of $\lambda$.

Let now  $u_\lambda$ denote the harmonic map from $D$ to $B(p,M)$
with boundary values  $\omega(\cdot,\lambda)$. In particular,
$m(\lambda):= \inf_{z\in D} |J(u_\lambda)(z)|$ depends
continuously on $\lambda$. We define $L:= \{\lambda \in [0,1]:
m(\lambda) > 0 \}$.  $0 \in L$; ($u_0$ is the conformal map
$\phi$), and therefore $L$  is not empty.

By Proposition \ref{p7.2},  $m(\lambda)\geq m_0 > 0$  for $\lambda
\in L$. Since $m(\lambda)$   depends continuously on $\lambda$,
(8.1.7) implies  $L = [0,1]$. Thus,  $u_1$  is a local
diffeomorphism and a diffeomorphism between the boundaries of $D$
and $u_1(D)$, and consequently a global diffeomorphism by
topology.
%(8.1)
Theorem \ref{8.1}  and the uniqueness theorem of Jager and Kaul
(cf. Theorem  5.1 \cite{jost2}) imply

\begin{cor}
Under the assumptions of Thm. \ref{8.1}, each harmonic map which
solves the Dirichlet problem defined by $g$   and which maps
$\Omega$  into a geodesic disc $B(p,M)$  with radius $M < \pi
/2k$, is a diffeomorphism  in  $\Omega$.
\end{cor}

%In section 8.1,

The above, we have assumed that the boundary of the image is
strictly convex, and, in addition, that the boundary values are a
diffeomorphism of class $C^2$.
%In this section, we shall prove t
The theorem  also holds for the case that\\
(f1): the boundary is only supposed to be convex and that\\
(f2): the boundary values are only supposed to induce a
homeomorphism of the boundaries.\\
%We shall present only
The procedure to handle  the case (f1)  is called the first
approximation argument. It is a modification of the corresponding
one given by E. Heinz.

The case (f1) of a general boundary is handled by an approximation
by smooth curves.
%Therefore, let us suppose
More precisely  it is supposed that the boundary of the image
$f(\mathbb{D})$ is only convex, while the boundary values $f$ are
still assumed to be a diffeomorphism of class $C^2$. For
approximation arguments  in planar case see also
\cite{kal.studia,Iw14,IwOn}.

{\bf Acknowledgement}.
\smallskip I would like to thank V. Bozin   and  D.  Kalaj  for  stimulating discussions related to this topic.
%After writing this paper V. Bozin informed  me  about \cite{man}
%{\rm (b)}.  It seems that there is some overlap between results
%presented in  Section \ref{grad3}   and  \cite{man} {\rm (b)}, concerning  univalent harmonic K-qc gradient mappings.
In particular,  D. Kalaj turned my attention on important
consideration in \cite{jost2} and informed  me  about  Iwaniec
-Onninen paper \cite{IwOn}.
%%%%%%   for reading  several versions of the manuscript  and

%U ovom radu http://elib.mi.sanu.ac.rs/files/journals/publ/96/n090p003.pdf (publicationes kalaj 2004), dokazao sam
%PogledajteÂ

{}

\end{document}